\documentclass[12pt,reqno]{amsart}
\usepackage{amsmath,amsthm,amsfonts,amssymb,bm,euscript}
\usepackage{fancyhdr,amscd}
\usepackage{anysize}
\usepackage{graphicx,epsfig}
\usepackage{amsthm,latexsym,amsmath,amssymb}
\usepackage{mathrsfs}
\usepackage{cite}
\usepackage{indentfirst}
 \usepackage[titletoc]{appendix}
 \usepackage [latin1]{inputenc}
 
\numberwithin{equation}{section}
\def\R{{\mathbb R}}

\def\u{u}

\usepackage{setspace}
\allowdisplaybreaks

\newcommand{\beq}{\begin{equation}}
\newcommand{\eeq}{\end{equation}}
\newcommand{\ben}{\begin{eqnarray}}
\newcommand{\een}{\end{eqnarray}}
\newcommand{\beno}{\begin{eqnarray*}}
\newcommand{\eeno}{\end{eqnarray*}}

\newcommand{\pa}{\partial}

\newtheorem{theorem}{\textbf Theorem}[section]
\newtheorem{lemma}{\textbf Lemma}[section]

\newtheorem{prop}{\textbf Proposition}[section]
\newtheorem{defin}{\textbf Definition}[section]

\newtheorem{remark}{\it Remark}

\allowdisplaybreaks \voffset=-0.2in \numberwithin{equation}{section}

\makeatletter
\@namedef{subjclassname@2020}{%
	\textup{2020} Mathematics Subject Classification}
\makeatother

\voffset=-0.2in \numberwithin{equation}{section}\allowdisplaybreaks

\begin{document}

\title[2D MHD equation with horizontal dissipation]
{Global regularity for the 2D MHD equations with horizontal dissipation and horizontal magnetic diffusion}

\author[M. Paicu, N. Zhu]{Marius Paicu$^{1}$ and
	Ning Zhu$^{2}$}

\address{$^1$ Universit\'e de Bordeaux, Institut de Math\'ematiques de Bordeaux, F-33405 Talence Cedex, France}

\email{marius.paicu@math.u-bordeaux.fr}

\address{$^2$ School of Mathematical Sciences, Peking University, 100871, Beijing, P. R. China }

\email{mathzhu1@163.com}
\begin{spacing}{1.5}

\begin{abstract} {This paper establishes the global regularity of classical solution to the 2D MHD system with only horizontal dissipation and horizontal magnetic diffusion in a strip domain $\mathbb{T}\times\R$ when the initial data is suitable small. To prove this, we combine the Littlewood-Paley decomposition with anisotropic inequalities to establish a crucial commutator estimate. We also analysis the asymptotic behavior of the solution.  In addition, the global existence and uniqueness of classical solution is obtained for the 2D simplified tropical climate model with only horizontal dissipations.}

\end{abstract}
\maketitle

\begin{section}{introduction}

The magnetic-hydrodynamics (MHD) system describe the motion of electrically conducting fluids such as plasmas, liquid metals, and electrolytes. They consist of a coupled system of the Navier-Stokes equations of fluid dynamics and Maxwell's equations of electromagnetism.
The MHD equations underlie many phenomena, such as the geomagnetic dynamo in geophysics,
solar winds and solar flares in astrophysics (see, e.g., \cite{Biskamp,davidson,pm}).	
	
The general 2D anisotropic MHD system can be written as follows:
\begin{equation}
\label{MHD1}
\left\{\begin{array}{l}
\partial_{t} u+(u \cdot \nabla) u-\nu_1\pa_1^2u-\nu_2\pa_2^2u=-\nabla P+(b \cdot \nabla) b, \quad (x,y)\in\Omega,~t\in\R_+,\\
\partial_{t} b+(u \cdot \nabla) b-\eta_1\pa_1^2b-\eta_2\pa_2^2b=(b \cdot \nabla) u, \\
\nabla \cdot u=0, \nabla \cdot b=0, \\
u(x, y, 0)=u_{0}(x, y), b(x, y, 0)=b_{0}(x, y).
\end{array}\right.
\end{equation}
where $u=u(x, y, t)$ denotes the velocity field, $P=P(x, y, t)$ is a scalar representing the pressure, $b=b(x, y, t)$ denotes the magnetic field of the fluid. The parameters 
$\nu_1,\nu_2\geq 0$ stand for the viscosity coefficient and $\eta_1,\eta_2\geq 0$ are the coefficient of the magnetic diffusion.

One of the most fundamental problems associate with the MHD equations is whether the global well-posedness result can be established. Although the global regularity results are obtained for both Navier-Stokes and Euler equations in two dimensions, it becomes much difficult for MHD due to the strong nonlinear coupling between the velocity equation and the magnetic induction equation. In the case of $\Omega=\R^2$, there are many results concerning this model. When (\ref{MHD1}) with Laplace dissipation in both velocity and magnetic equations, it is not hard to show that (\ref{MHD1}) possesses
a unique global solution with smooth initial data (see, e.g., \cite{dl,st}). If we ignore all of the dissipation ($\nu_1,\nu_2,\eta_1,\eta_2=0$), then (\ref{MHD1}) becomes the completely inviscid MHD and the global existence and regularity problems remain outstandingly open. In order to understand this problem, many mathematicians investigated  the intermediate case and a lot of significant progress has been made. For the case that \eqref{MHD1} with only Laplace dissipation in velocity equation ($\nu_1=\nu_2=\nu>0$, $\eta_1=\eta_2=0$), which called non-resistive MHD, Jiu and Niu \cite{JN2006} proved the local well-posedness in Soboble space $H^s$ with $s\geq3$. Later, Fefferman et. al. were able to weaken the regularity assumption to the initial data $(\u_0, b_0)\in H^s$ with $s>\frac d2$ in \cite{FMRR14} and then to $u_0 \in H^{s-1+\varepsilon}$ and $b_0\in H^s$ with $s>\frac d2$ in \cite{FMRR17}. Then Chemin et. al. in \cite{CMRR16} further improved the assumption to the Besov space $u_0\in B^{\frac d2-1}_{2,1}$ and $b_0\in B^{\frac d2}_{2,1}$, they obtained the local existence for $d=2$ and $d=3$, and uniqueness for $d=3$. The uniqueness of the case $d=2$ is obtained by Wan in \cite{Wan16}. Recently, Li et. al. made an important progress by reducing the assumption for initial data to the homogeneous Besov space $u_0\in \dot{B}^{\frac dp-1}_{p,1}$ and $b_0\in \dot{B}^{\frac dp}_{p,1}$ with $p\in[1, 2d)$~(see \cite{LTY17}). For the resistive MHD $(\nu_1=\nu_2=0, \eta_1=\eta_2=\eta>0),$ the global regularity is difficult even in the 2D case. In fact, whether or not resistive MHD equations alway possesses global smooth solutions is unknown no matter how smooth the initial data are. Some related work have been made in the past few years to understand this problem. If we replace the dissipation $-\eta\Delta b$ to fractional Laplacian $\eta(-\Delta)^\beta b$ with $\beta>1$, Cao et. al. in \cite{CWY14}, Jiu and Zhao in \cite{JZ15} obtain the global well-posedness via different approaches. Other interesting results on the resistive and non-resistive MHD systems can be found in \cite{CL18,HXY18,LXZ15,PZZ18,RWXZ14,RXZ16,WWX15,Yamazaki14}.

There are also some papers concerning about \eqref{MHD1} with partial dissipation. Besides the great interest in mathematics, the MHD equations with only partial viscosity and partial magnetic diffusion also have wide physical applicability. The partial dissipation turn out in certain physical regimes and after some suitable scaling. In \cite{CW11}, Cao and Wu studied \eqref{MHD1} with mixed dissipation that $\nu_1=0, \nu_2>0, \eta_1>0, \eta_2=0$, or $\nu_1>0, \nu_2=0, \eta_1=0, \eta_2>0$ in $\R^2$, they obtained the global regularity result for $H^2$ initial data. For \eqref{MHD1} with horizontal viscosity and horizontal magnetic diffusion ($\nu_1>0, \nu_2=0, \eta_1>0, \eta_2=0$), Cao et. al. showed some a priori estimates and gained the global regularity result under some extra condition (see \cite{CRW13}). In \cite{DZ15}, Du and Zhou investigated the anisotropic MHD with some mixed dissipations. Other interest results can be found in \cite{DJLW18,DLW19,LJWY20,WZ19} and the references therein.
	
In this paper, we investigate the following 2D MHD system with only horizontal viscosity and horizontal megnetic diffusion,
\begin{equation}
\label{system-MHD}
\left\{\begin{array}{l}
{\partial_{t} u+u\cdot \nabla u-\nu\partial_1^{2} u+\nabla p=b \cdot \nabla b}, \quad (x,y)\in\Omega,~t\in\R_+, \\
{\partial_t b+u\cdot \nabla b-\eta\partial_1^{2} b=b \cdot \nabla u},\\
{\nabla \cdot u=0,~~\nabla \cdot b=0},\\
u(x, y, 0)=u_0(x, y),~~b(x, y, 0)=b_0(x, y).
\end{array}\right.
\end{equation}
The spatial domain $\Omega$ here is taken to be
$$\Omega\triangleq \mathbb{T}\times \R,$$
where $\mathbb{T}=[0,1]$ being the periodic box and $\R$ being the whole line. Comparing with \cite{CRW13}, here the domain in $x$ is $\mathbb{T}$ rather then $\R$,
which can provide us some kind of Poincar\'e inequality. The first theorem following shows that \eqref{system-MHD} is global well-posedness under some smallness condition.  

\begin{theorem}
\label{thm1}
Assume $u_0,b_0\in H^s(\Omega)$ be two divergence-free vector fields with $s\geq2$ and satisfy the smallness condition  
\begin{equation}
\label{smallness condition}
\|(u_0,b_0)\|_{L^2}^2+\|(\pa_2u_0,\pa_2b_0)\|_{L^2}^2\leq \delta,
\end{equation}
for some $\delta>0$. Then system \eqref{system-MHD} with the initial data $(u_0,b_0)$ has a unique global solution $(u,b)$. In addition, $(u,b)$ satisfies
\begin{equation}
u,b\in L^\infty(\R_+;H^s(\Omega)),~\pa_1u,\pa_1b\in L^2(\R_+;H^s(\Omega)).
\end{equation}
\end{theorem}

If we separate the function $(u, b)$ into the $x$-average part $(\bar{u}, \bar{b})$ and the remainder part $(\widetilde{u}, \widetilde{b})$, which are defined by
\begin{equation*}
\begin{split}
&\bar{u}(y)=\int_{\mathbb{T}}u(x',y)~dx'\quad\text{and}\quad\widetilde{u}(x,y)=u(x,y)-\bar{u}(y),\\
&\bar{b}(y)=\int_{\mathbb{T}}b(x',y)~dx'\quad\text{and}\quad\widetilde{b}(x,y)=b(x,y)-\bar{b}(y).
\end{split}
\end{equation*}
Then we can obtain the next theorem which showing that the remainder part $(\widetilde{u}, \widetilde{b})$ decay exponentially in time.

\begin{theorem}
\label{thm2}
Let $(u_0,b_0)$ satisfies the assumptions in Theorem \ref{thm1}, then 
\begin{equation*}
\|(\widetilde{u},\widetilde{b})(t)\|_{L^2}\leq Ce^{-\widetilde{C}t},
\end{equation*}
where $C,\widetilde{C}$ are two positive constant depending on $\nu,\eta$ and $(u_0,b_0)$. 
\end{theorem}

\begin{remark}
This theorem indicate that the solution $u$ will tend to its  $x$-average $\bar{u}$ as $t\rightarrow \infty$. The same fact also holds for $b$.
\end{remark}

In this paper, we are also concerned
with the tropical climate model in the simplified case where the temperature $\theta$ is a constant. 
The classical tropical climate model can be read as:

\begin{equation}
\label{tropical-model}
\begin{cases}
\partial_t u^\epsilon +u^\epsilon\cdot \nabla u^\epsilon +\text{div\,}(v^\epsilon \otimes v^\epsilon )-\nu\Delta u^\epsilon+\nabla p^\epsilon =0\\
\partial_t v^\epsilon +u^\epsilon\cdot \nabla v^\epsilon  -\eta\Delta v^\epsilon+\frac{1}{\epsilon}\nabla \theta^\epsilon =-v^\epsilon \cdot\nabla u^\epsilon\\
\partial_t\theta^\epsilon +u^\epsilon \cdot\nabla\theta^\epsilon +\frac{1}{\epsilon}\text{div\,} v^\epsilon =0\\
\text{div\,} u^{\epsilon}=0,
\end{cases}
\end{equation}
where the vector fields $u^{\epsilon}=(u^{\epsilon1},u^{\epsilon2})$ and $v^{\epsilon} = (v^{\epsilon1}, v^{\epsilon2})$ denote the barotropic mode and the first baroclinic mode of the velocity, respectively. The scalar $p^{\epsilon}$ denotes the pressure and $\theta^{\epsilon}$ the temperature. The parameters $\nu,\eta\geq 0$ denote the viscosity and $\epsilon$ is proportional to the Froude number. In physical situations, $\epsilon$ normally suppose to be very small \cite{Majda07,MK03}. We remark that the tropical climate model bears some similarities but is different from
the MHD equations \eqref{MHD1}.
A more subtle difference is that
$b\cdot\nabla u$ in \eqref{MHD1} has a negative sign while $v^{\varepsilon}\cdot \nabla u^{\varepsilon}$ has a positive sign in \eqref{tropical-model}. This sign makes
a difference in the study of the global existence and regularity problem on \eqref{tropical-model}.

Equations \eqref{tropical-model} without dissipation was
derived by Frierson, Majda and Pauluis for large-scale dynamics of precipitation fronts
in the tropical atmosphere \cite{FMP04}. Its viscous counterpart with the standard Laplacian
can be derived by the same argument from the viscous primitive equations (see \cite{LT16-1}). Later on, many literatures
investigated \eqref{tropical-model} with partial or fractional Laplacian which can be found in \cite{DWWYZ19,DWWZ19,DWY19,LT16-2} and the reference therein.

By passing to the limit when $\epsilon$ is converging to zero, we formally obtain the following simplified system
\begin{equation}
\label{TC}
\begin{cases}
\partial_t u +u\cdot \nabla u -\nu\Delta u+\nabla p =-v\cdot\nabla v,\\
\partial_t v +u\cdot \nabla v  -\eta\Delta v+\nabla \Phi =-v \cdot\nabla u,\\
\text{div\,} u=0,~\text{div\,} v=0.
\end{cases}
\end{equation}
Indeed, according to the classical energy estimates, we have $(u^\varepsilon,v^\varepsilon)\in L^\infty([0,T];L^2(\Omega))\cap L^2([0,T];H^1(\Omega))$, $\theta^{\varepsilon}\in L^\infty([0,T];L^2(\Omega))$. Then one can deduce $(u^{\varepsilon},v^{\varepsilon})$ convergence weakly to its limit $(u,v)$ in $L^\infty([0,T];L^2(\Omega))\cap L^2([0,T];H^1(\Omega))$ and strongly in $L^2([0,T];L^2(\Omega))$ by compact embedding. $\theta^{\varepsilon}$ convergence weakly to $\theta$ in $L^\infty([0,T];L^2(\Omega))$.

According to the third equation of system \eqref{tropical-model},
$$\epsilon(\partial_t\theta^\epsilon+u^\epsilon\cdot\nabla\theta^\epsilon)+\text{div\,} v^\epsilon=0.$$
For any $\varphi\in C^{\infty}_0 ([0,T];\Omega)$,
$$\epsilon\int_0^T\int_{\Omega}(\theta^\epsilon\cdot\partial_t\varphi+ \theta^\epsilon u^\epsilon\cdot \nabla\varphi )~dxdt+\int_0^T\int_{\Omega} v^\epsilon\cdot \nabla\varphi~ dxdt=0.$$
By passing the limit $\varepsilon\rightarrow 0$, we have
$$\int_0^T\int_{\Omega} v\cdot \nabla\varphi~dxdt=0,$$
which implies $\text{div\,} b=0$ in the sense of distributions.

On the other hand, by taking the $\text{div}$ on the equation of $v^\epsilon$ in system \eqref{tropical-model}, we get
$$\partial_t \text{div\,} v^\epsilon +\text{div\,}(u^\epsilon\cdot\nabla v^\epsilon+v^\varepsilon\cdot\nabla u^\varepsilon)-\eta\Delta \text{div\,} v^\epsilon+\frac{1}{\epsilon}\Delta \theta^\epsilon=0,$$
and we get $\frac 1\epsilon \nabla \theta^\epsilon$ converge to $\nabla\Phi\triangleq-\nabla \Delta^{-1}[\partial_t \text{div\,} v +\text{div\,}(u\cdot\nabla v+v\cdot\nabla u)-\eta\Delta \text{div\,} v]$ in the sense of distributions. 

Taking the limit $\varepsilon\rightarrow 0$ in the second equation of \eqref{tropical-model} and replacing $\frac{1}{\varepsilon}\nabla\theta^{\varepsilon}$ by $\nabla \Phi$, we get the second equation of \eqref{TC}.

Here we consider system \eqref{TC} with only horizontal dissipation which can be read as:
 \begin{equation}
 \label{system-tropical}
 \begin{cases}
 \partial_t u +u\cdot \nabla u -\nu\pa_1^2 u+\nabla p =-v\cdot\nabla v,\\
 \partial_t v +u\cdot \nabla v  -\eta\pa_1^2 v+\nabla \Phi =-v \cdot\nabla u,\\
 \text{div\,} u=0,~\text{div\,} v=0,\\
u(x, y, 0)=u_0(x, y),~~v(x, y, 0)=b_0(x, y).
 \end{cases}
 \end{equation}
Comparing with \eqref{system-MHD}, the forcing term of system \eqref{system-tropical} has different sign, which give us more cancellation in the energy estimates. The main result can be stated as follows. 

\begin{theorem}
\label{thm1-TC}
	Assume $u_0,v_0\in H^s(\Omega)$ be two divergence-free vector fields with $s\geq2$. Then system \eqref{system-tropical} with the initial data $(u_0,v_0)$ has a unique global solution $(u,v)$. In addition, $(u,v)$ satisfies
	\begin{equation}
	u,v\in L^\infty([0,T];H^s(\Omega)),~\pa_1u,\pa_1v\in L^2([0,T];H^s(\Omega)),
	\end{equation}
for any $T>0$.
\end{theorem}
\begin{remark}
The result in Theorem \ref{thm1-TC} is also true for the case $\Omega=\R^2$.
\end{remark}

The rest of this paper is divided into six sections and an appendix. The second section gives the introduction of Littlewood-Paley decomposition and Besov space in $\Omega$, and proves some properties of the functions $\bar{u}$ and $\widetilde{u}$. In the third section, we establish the global $L^2$ estimates for $(u,b)$ and $(\pa_2u,\pa_2b)$. In section 4, we provide the regularity estimate for $(u,b)$ and section 5 completes the proof of Theorem \ref{thm1}. In section 6, we present the proof of Theorem \ref{thm2}. In the last section, we give the proof of Theorem \ref{thm1-TC}. In Appendix A, we provides the details of proving Lemma \ref{three-linear term}, which is a crucial tool to deal with the convection term in the energy estimates. 

Let us end this section by the notations that we shall use in this context.

Through out this paper, $C$ stands for some real positive constant which may vary from line to line. $\left\{b_{q}\right\}$ stands for a generic sequence in $\ell^{1}$ which may be different in each
occurrence. Here, we have denotes the $\ell^{1}$ space of summable sequences with the norm $\|\left\{b_{q}\right\}_{q}\|_{\ell^{1}}=\sum_{q}\left|b_{q}\right|$. $\|f(x,y)\|_{L^r}$ denotes the norm in the isotropic Lebesgue space $L^r(\Omega)$ while $\|f(x,y)\|_{L^q_yL^p_x}:=\bigg\|\|f(x,y)\|_{L^p_x}\bigg\|_{L^q_y}$ denotes the norm in the anisotropic Lebesgue space $L^q_yL^p_x(\Omega)\triangleq L^q_y(\R; L^p_x(\mathbb{T}))$.

\begin{section}{Preliminaries}
In this section, we shall first introduce the Littlewood-Paley theory in $\Omega$, which will be used frequently in the rest of this paper. And then prove some important properties of the function $\bar{u}$ and $\widetilde{u}$. Before starting the main body, we recall the definitions of the Fourier transform and inverse Fourier transform in $\Omega$:
\begin{equation*}
\mathcal{F}f(k,\xi)=\widehat{f}(k,\xi)=\frac{1}{2\pi}\int_{\Omega}f(x,y)e^{-i(xk+y\xi)}~dxdy,
\end{equation*}
\begin{equation*}
\mathcal{F}^{-1}f(x,y)=\check{f}(x,y)=\sum_{k=-\infty}^{\infty}\int_{\R}f(k,\xi)e^{i(xk+y\xi)}~d\xi.
\end{equation*}

Next, we present the Littlewood-Paley theory in $\Omega$ which plays an important role in the proof of our results. More details about Littlewood-Paley theory and related materials can be found in \cite{BCD11} for whole space situation and \cite{Danchin05,DHWX20} for periodic case. 

Let $\chi(\eta,\xi)$ be a smooth function supported on the ball $\mathcal{B} \triangleq\left\{(\eta,\xi) \in \R^2:\sqrt{\eta^2+\xi^2} \leq \frac{4}{3}\right\}$ and $\varphi(\eta,\xi)$ be a smooth function supported on the ring $\mathcal{C} \triangleq\left\{(\eta,\xi) \in \R^2: \frac{3}{4} \leq \sqrt{\eta^2+\xi^2} \leq \frac{8}{3}\right\}$ such that
\begin{equation*}
\chi(\eta,\xi)+\sum_{q \geq 0} \varphi\left(2^{-q}\eta, 2^{-q} \xi\right)=1, \quad \text { for all } \quad (\eta,\xi) \in \R^2,
\end{equation*}
\begin{equation*}
\sum_{q \in \mathbb{Z}} \varphi\left(2^{-q}\eta,2^{-q} \xi\right)=1, \quad \text { for all } \quad (\eta,\xi) \in \R^2 \backslash\{0\}.
\end{equation*}
Then for every $u \in \mathcal{S}^{\prime}$ (temperate distribution on $\Omega$), we define the non-homogeneous Littlewood-Paley operators as follows,
\begin{equation*}
\Delta_{q} u=0 \text { for } q \leq-2,
\end{equation*}
\begin{equation*}
\begin{split}
\Delta_{-1} u\triangleq\chi(D) u=\mathcal{F}^{-1}(\chi(k,\xi) \widehat{u}(k,\xi))
=\int_{\mathbb{T}}\int_{\R} \widetilde{h}(\mathbf{x}')u(\mathbf{x}-\mathbf{x}')~d\mathbf{x}',
\end{split}
\end{equation*}
with
\begin{equation}
\begin{split}
\widetilde{h}(\mathbf{x}')=\widetilde{h}(x,y)=\frac{1}{2\pi}\sum_{k\in\mathbb{Z}}\int_{\R}\chi(k,\xi)e^{i(kx+\xi y)}d\xi,
\end{split}
\end{equation}
and
\begin{equation*}
\begin{split}
\Delta_{q} u\triangleq\varphi\left(2^{-q} D\right) u=\mathcal{F}^{-1}\left(\varphi\left(2^{-q}k, 2^{-q}\xi\right) \widehat{u}(k,\xi)\right)
=\int_{\mathbb{T}}\int_{\R} {h_q}(\mathbf{x}')u(\mathbf{x}-\mathbf{x}')~d\mathbf{x}', \quad \forall q \geq 0, 
\end{split}
\end{equation*}
where
\begin{equation}\label{hq}
\begin{split}
{h_q}(\mathbf{x}')={h_q}(x,y)=\frac{1}{2\pi}\sum_{k\in\mathbb{Z}}\int_{\R}\varphi(2^{-q}k,2^{-q}\xi)e^{i(kx+\xi y)}d\xi.
\end{split}
\end{equation}
And we also define the low frequency cut-off operator:
\begin{equation*}
\begin{split}
S_{q} u=\sum_{j=-1}^{q-1} \Delta_{j} u.
\end{split}
\end{equation*}
With these operators, we can define the Littlewood-Paley decomposition, for any $u\in\mathcal{S}^{\prime}(\Omega)$,
\begin{equation*}
u=\sum_{q=-1}^{\infty} \Delta_{q} u.
\end{equation*}
In terms of the operators $\Delta_{q}$ and $S_{q},$ we can write a standard product of two functions as a sum of paraproducts, which is called Bony's decomposition as in the whole space case (see, e.g., \cite{BCD11,Bony81})
$$
f g=T_{f} g+T_{g} f+R(f, g)
$$
where
$$
T_{f} g=\sum_{q} S_{q-1} f \Delta_{q} g, \quad R(f, g)=\sum_{q} \sum_{k>q-1} \Delta_{k} f \tilde{\Delta}_{k} g
$$
with $\widetilde{\Delta}_{k}=\Delta_{k-1}+\Delta_{k}+\Delta_{k+1}$.

Next, we state the definition of non-homogeneous Besov spaces in $\Omega$ through the dyadic decomposition.

\begin{defin}
For $s \in \mathbb{R}$ and $1 \leq p, r \leq \infty$, the non-homogeneous Besov space $B_{p, r}^{s}$ in $\Omega$ is defined by
$$
B_{p, r}^{s}=\left\{f \in \mathcal{S}^{\prime} ;\|f\|_{B_{p, r}^{s}}<\infty\right\},
$$
where
\begin{equation*}
\begin{split}
\|f\|_{B_{p, r}^{s}}=\left\{\begin{array}{l}
\sum_{q \geq-1}\left(2^{q s}\left\|\Delta_{q} f\right\|_{L^{p}}^{r}\right)^{\frac{1}{r}}, \text { for } r<\infty, \\
\sup _{q \geq-1} 2^{q s}\left\|\Delta_{q} f\right\|_{L^{p}}, \quad \text { for } r=\infty.
\end{array}\right.	
\end{split}
\end{equation*}
\end{defin}
\noindent
We point out that when $p=r=2,$ for all $s \in \mathbb{R},$ we have $B_{2,2}^{s}\left(\Omega\right)=H^{s}\left(\Omega\right).$

The following Bernstein type inequalities are useful tools on Fourier localized functions, as these inequalities trade integrability for derivatives.
\begin{lemma}[Bernstein inequality]
Let $k \in \mathbb{N} \cup\{0\}$, $1 \leq a \leq b \leq \infty .$ \\
(1)~There exists a constant $C_{1}$ such that for some integer $q,$ 
$$
\left\|\nabla^{\alpha} \Delta_qf\right\|_{L^{b}} \leq C_{1} 2^{q\left(k+2\left(\frac{1}{a}-\frac{1}{b}\right)\right)}\|\Delta_qf\|_{L^{a}}, \quad k=|\alpha|,
$$
and
$$
\left\|S_qf\right\|_{L^{b}} \leq C_{1} 2^{2q\left(\frac{1}{a}-\frac{1}{b}\right)}\|S_qf\|_{L^{a}}.
$$
(2)~There exist two positive constants $C_{2},C_{3}$ (depending on $a$ and $b$) such that, for any integer $q \geq 0$,
\begin{equation*}
C_{2} 2^{q k}\|\Delta_qf\|_{L^{b}} \leq\left\|\nabla^{\alpha} \Delta_qf\right\|_{L^{b}} \leq C_{3} 2^{q\left(k+2\left(\frac{1}{a}-\frac{1}{b}\right)\right)}\|\Delta_qf\|_{L^{a}}, \quad k=|\alpha|.
\end{equation*}
\end{lemma}

For a general function $f:\Omega\rightarrow\R$, we define
\begin{equation}
\label{decomposition}
\bar{f}(y)=\int_{\mathbb{T}}f(x',y)~dx'\quad\text{and}\quad\widetilde{f}(x,y)=f(x,y)-\bar{f}(y).
\end{equation}
It is easy to check that $\bar{f}$ and $\widetilde{f}$ satisfy the following properties:
\begin{lemma}
\label{properties1}
Let $\bar{f}$ and $\widetilde{f}$ defined as \eqref{decomposition}, then
\begin{flalign}
&1.~\widetilde{f}~\text{is zero-average in}~x, \text{which is}~\int_{\mathbb{T}}\widetilde{f}(x',y)~dx'=0.& \label{zero-average}
\\
&2.~\widetilde{f}~\text{satisfies Poincar\'e inequality in}~x, \text{which is}~\|\widetilde{f}\|_{L^2_x}\leq C\|\pa_1\widetilde{f}\|_{L^2_x}.&\label{Poincare}
 \\
&3.~\widetilde{f}~\text{satisfies the interpolation}~\|\widetilde{f}\|_{L^\infty_x}\leq C\|\widetilde{f}\|_{L^2_x}^{\frac12}\|\pa_1\widetilde{f}\|_{L^2_x}^{\frac12}.&\label{interpolation}
\end{flalign}
\end{lemma}
For the divergence-free vector field $u$, we have the following special properties for $\bar{u}$ and $\widetilde{u}$.
\begin{lemma}
\label{properties2}
Let $u$ be a smooth divergence-free vector field, $\bar{u}$ and $\widetilde{u}$ are defined as \eqref{decomposition}, then we have the following properties:
\begin{flalign}
&1.~\pa_2\bar{u}^2=0~\text{and}~\bar{u}^2=0.&\label{u2=0}\\
&2.~\widetilde{u}~\text{also satisfies the divergence-free condition}~\pa_1\widetilde{u}^1+\pa_2\widetilde{u}^2=0. &\label{divergence-free-2}
\end{flalign}
\end{lemma}
\begin{proof}
By the divergence-free condition 
\begin{equation*}
\pa_1u^1+\pa_2u^2=0,
\end{equation*}
then integrating in $x$ over $\mathbb{T}$,
\begin{equation*}
\int_{\mathbb{T}}\pa_1u^1(x',y)~dx'+\int_{\mathbb{T}}\pa_2u^2(x',y)~dx'=0.
\end{equation*}
Thus we obtain 
\begin{equation*}
\pa_2\bar{u}^2=\pa_2\int_{\mathbb{T}}u^2(x',y)~dx'=-\int_{\mathbb{T}}\pa_1u^1(x',y)~dx'=0.
\end{equation*}
Because of $\lim_{y\rightarrow\pm\infty}\bar{u}^2(y)=0$, then we deduce $\bar{u}^2=0$, which verifies (2.5).\\
For (2.6), according to the definition of $\widetilde{u}$,
\begin{equation*}
\pa_1\widetilde{u}^1+\pa_2\widetilde{u}^2=\pa_1{u}^1-\pa_1\bar{u}^1+\pa_2{u}^2-\pa_2\bar{u}^2.
\end{equation*}
Then combining with the result (2.5), we can obtain
\begin{equation*}
\pa_1\widetilde{u}^1+\pa_2\widetilde{u}^2=\pa_1{u}^1+\pa_2{u}^2=0,
\end{equation*}
which completes the proof of this lemma.
\end{proof}

\end{section}

\begin{section}{Global $L^2$ bound for $(u,b)$ and $(\pa_2u,  \pa_2 b)$}

In this section, we will give the global a priori estimates for $(u,b)$ and $(\pa_2u,\pa_2b)$ in $L^2$ space. The main result can be stated as the following proposition.
\begin{prop}
\label{prop-2.1}
Assume that $(u_0,b_0)\in L^2$, $(\pa_2u_0, \pa_2b_0)\in L^2$ and satisfies the smallness condition 
\begin{equation}
\|(u_0,b_0)\|_{L^2}^2+\|(\pa_2u_0,\pa_2b_0)\|_{L^2}^2\leq \delta,
\end{equation}
for some $\delta>0$. Let $(u,b)$ be the corresponding solution of \eqref{system-MHD}, then for all $t>0$, we have 
\begin{equation*}
\|(u,b)(t)\|_{L^2}^2+\|\pa_2(u,b)(t)\|_{L^2}^2
+\overline{C}\int_0^t\|\pa_1(u,b)(\tau)\|_{L^2}^2 d\tau+\overline{C}\int_0^t\|\pa_1\pa_2(u,b)(\tau)\|_{L^2}^2 d\tau \leq C_0\delta,
\end{equation*}
where $\overline{C}=\min\{\nu,\eta\}$ and $C_0>0$ be a constant independent of time $t$. 	 
\end{prop}

\end{section}
\begin{proof}
Taking $L^2$ inner product of \eqref{system-MHD} with $(u,b)$, one can obtain:
\begin{equation}
\frac{1}{2} \frac{d}{d t}(\|u(t)\|_{L^{2}}^{2}+\| b(t)\|_{L^2}^2)+\nu\|\partial_{1} u\|_{2}^{2}+\eta\|\partial_{1} b\|_{2}^{2}=0.
\end{equation}
Then integrating in time, we get the $L^2$ estimate for $(u,b)$ that
\begin{equation}
\label{L2-ub}
\|u(t)\|_{2}^{2}+\|b(t)\|_{2}^{2}+2\nu\int_{0}^{t}\|\partial_1 u(\tau)\|_{L^{2}}^{2} d \tau+2\eta\int_{0}^{t}\|\partial_1 b(\tau)\|_{L^2}^{2} d{\tau} \leq\|u_{0}, b_{0}\|_{2}^{2}.
\end{equation}
Next we derive the $L^2$ estimate for $\pa_2(u,b)$. Applying $\pa_2$ to system \eqref{system-MHD},
\begin{equation}
\label{pa_2-MHD}
\left\{\begin{array}{cc}\begin{split}
&\partial_t \partial_{2} u+u \cdot \nabla \partial_{2} u+\pa_2 u \cdot \nabla u-\nu\partial_{1}^{2} \partial_{2} u+\nabla \partial_{2} p=b \cdot \nabla \pa_{2} b+\partial_{2} b \cdot \nabla b, \\
&\partial_{t} \partial_{2} b+u \cdot \nabla\partial_2 b+\partial_{2}u \cdot \nabla b-\eta\partial_{1}^{2} \partial_{2} b=b \cdot \nabla\pa_{2} u+\pa_2 b \cdot \nabla u.\end{split}
\end{array}\right.
\end{equation}
Multiplying \eqref{pa_2-MHD} by $(\pa_2u,\pa_2b)$ and integrating over $\Omega$, because of the divergence-free condition of $u$ and $b$, we deduce
\begin{equation}
\label{L2-pa_2ub-1}
\begin{split}
&\quad\frac{1}{2} \frac{d}{d t}\left(\|\pa_2u(t)\|_{L^{2}}^{2}+\|\pa_2b(t)\|_{L^{2}}^{2}\right)+\nu\|\partial_{1}\pa_2 u\|_{2}^{2}+\eta\|\partial_{1}\pa_2 b\|_{2}^{2}\\
&=-\int_{\Omega}\pa_2u\cdot\nabla u\cdot\pa_2 u~dxdy+\int_{\Omega}\pa_2b\cdot\nabla b\cdot\pa_2 u~dxdy\\
&\quad-\int_{\Omega}\pa_2u\cdot\nabla b\cdot\pa_2 b~dxdy+\int_{\Omega}\pa_2b\cdot\nabla u\cdot\pa_2 b~dxdy\\
&\triangleq A_1+A_2+A_3+A_4.
\end{split}
\end{equation}
For $A_{1},$ we can write it as
\begin{align*}
A_{1} &=-\int_{\Omega} \partial_{2} u \cdot \nabla u \cdot \partial_{2} u d x d y \\
&=-\int_{\Omega} \partial_{2} u^{1} \partial_{1} u^{1} \partial_{2} u^{1} d x d y-\int_{\Omega} \partial_{2} u^{1} \partial_{1} u^{2} \partial_{2} u^{2} d x d y \\
&\quad-\int_{\Omega} \partial_{2} u^{2} \pa_{2} u^{1} \partial_{2} u^{1} d x d y-\int_{\Omega} \partial_{2} u^{2} \partial_{2} u^{2} \partial_{2} u^{2} d x d y \\
&=-\int_{\Omega} \partial_{2} u^{1} \partial_{1} u^{2} \partial_{2} u^{2} d x d y-\int_{\Omega} \partial_{2} u^{2} \partial_2 u^{2} \partial_2 u^{2} d x d y\\
&\triangleq A_{11}+A_{12}.
\end{align*}
Using the decomposition \eqref{decomposition} and according to the property \eqref{u2=0}, $A_{11}$ can be written as
$$
\begin{aligned}
A_{11} &=-\int_{\Omega} \pa_{2} u^{1} \partial_{1} u^{2} \partial_{2} u^{2} d x d y \\
&=-\int_{\Omega} \partial_{2} \bar{u}^{1} \partial_{1} u^{2} \partial_{2} u^{2} d x d y-\int_{\Omega} \partial_{2} \widetilde{u}^{1} \partial_{1} u^{2} \partial_{2} u^{2} d x d y \\
&=-\int_{\Omega} \partial_{2} \bar{u}^{1} \partial_{1} \widetilde{u}^{2} \partial_{2} \widetilde{u}^{2} d x d y-\int_{\Omega} \partial_{2} \widetilde{u}^{1} \partial_{1} u^{2} \partial_{2} u^{2}  d x d y \\
& \triangleq A_{111}+A_{112}.
\end{aligned}.
$$
For $A_{111}$, by anisotropic
H\"older inequality, interpolation \eqref{interpolation}, Poincar\'e inequality \eqref{Poincare} and Young's inequality,
\begin{equation}
\begin{aligned}
A_{111} & \leq\left\|\partial_{2} \bar{u}^{1}\right\|_{L^2_y}\left\|\partial_{1} \widetilde{u}^{2}\right\|_{L^\infty_yL^{2}_x}\left\|\partial_{2}\widetilde{u} ^{2}\right\|_{L^{2}_x L_{y}^{2}} \\
& \leq\left\|\partial_{2} \bar{u}^{1}\right\|_{L^{2}_y}\left\|\partial_{1} \widetilde{u}^{2}\right\|_{L^2}^{\frac{1}{2}}\left\|\partial_{1} \partial_{2} \widetilde{u}^{2}\right\|_{L^2}^{\frac{1}{2}}\left\|\partial_{2} \widetilde{u}^{2}\right\|_{L^2} \\
& \leq C\left\|\partial_{2} \bar{u}^{1}\right\|_{L^{2}_y}\left\|\partial_{1} \widetilde{u}\right\|_{L^2}\left\|\partial_{1} \partial_{2} \widetilde{u}^{2}\right\|_{L^2}\\
& \leq C\left\|\partial_{2} u\right\|_{L^{2}}^{2}\left\|\partial_{1} u\right\|_{L^2}^{2}+\varepsilon\left\|\partial_{1} \partial_{2} u\right\|_{L^{2}}^{2}.
\end{aligned}
\end{equation}
Similarly, $A_{112}$ can be bounded by 
\begin{equation} 
\begin{split}
A_{112} &=-\int_{\Omega} \partial_{2} \widetilde{u}^{1} \partial_{1} u^{2} \partial_{2} u^{2} d x d y \\ & \leq\left\|\partial_{2} \widetilde{u}^{1 }\right\|_{L^\infty_xL^2_y}\left\|\partial_{1} u^{2}\right\|_{L^\infty_yL^2_x}\left\|\partial_{2} u^{2}\right\|_{L^2} \\ & \leq C\left\|\partial_{2} \widetilde{u}^{1}\right\|_{L^2}^{\frac{1}{2}}\left\|\partial_{1} \partial_{2} \widetilde{u}^{1}\right\|_{L^2}^{\frac{1}{2}}\left\|\partial_{1} u^{2}\right\|_{L^2}^{\frac{1}{2}}\left\|\partial_{1} \partial_{2} u^{2}\right\|_{L^2}^{\frac{1}{2}}\left\|\partial_{2} u^{2}\right\|_{L^2} \\ & \leq C\left\|\partial_{2} u\right\|_{L^2}^{2}\left\|\partial_{1} u\right\|_{L^2}^{2}+\varepsilon\left\|\partial_{1} \partial_{2} u\right\|_{L^2}^{2}.
\end{split} 
\end{equation}
Then because of the divergence-free condition of $u$, 
\begin{equation*}
A_{12}=-\int_{\Omega} \partial_{2} u^{2} \partial_2 u^{2} \partial_2 u^{2} d x d y=
\int_{\Omega} \partial_{2} u^{2} \partial_1 u^{1} \partial_2 u^{2} d x d y.
\end{equation*}
Making use of the same way as $A_{11}$, we can bound $A_{12}$ by
\begin{equation*} 
\begin{split}
A_{12}\leq C\left\|\partial_{2} u\right\|_{L^2}^{2}\left\|\partial_{1} u\right\|_{L^2}^{2}+\varepsilon\left\|\partial_{1} \partial_{2} u\right\|_{L^2}^{2}.
\end{split} 
\end{equation*}
Thus, we get the estimate of $A_{1}$,
\begin{equation} 
\label{A-1}
\begin{split}
A_{1}  \leq C\left\|\partial_{2} u\right\|_{L^2}^{2}\left\|\partial_{1} u\right\|_{L^2}^{2}+\varepsilon\left\|\partial_{1} \partial_{2} u\right\|_{L^2}^{2}.
\end{split} 
\end{equation}
Then we estimate $A_2$, because of the divergence-free condition of $b$,
\begin{align}
A_{2}&= \int_{\Omega} \pa_2 b \cdot \nabla b \cdot\partial_2 u  d x d y \notag\\
&= \int_{\Omega} \partial_{2} b^{1} \partial_{1} b^{1} \partial_{2} u^{1} d x d y+\int_{\Omega} \partial_{2} b^{1} \partial_{1} b^{2} \partial_{2} u^{2} d x d y \notag\\
&\quad+\int_{\Omega} \partial_{2} b^{2} \pa_2b^1 \partial_{2} u^{1} d {x} d y+\int_{\Omega} \partial_{2} b^{2} \partial_{2} b^{2} \partial_{2} u^{2} d x d y \notag\\
&=\int_{\Omega} \partial_{2} b^{1} \partial_{1} b^{2} \partial_{2} u^{2} d x d y+\int_{\Omega} \partial_{2} b^{2} \partial_{2} b^{2} \partial_{2} u^{2} d x d y \notag\\
&=A_{21}+A_{22}.
\end{align}
Similar as $A_{11}$, we decompose $A_{21}$ as 
$$
\begin{aligned}
A_{21} &=\int_{\Omega} \pa_{2} b^{1} \partial_{1} b^{2} \partial_{2} u^{2} d x d y \\
&=\int_{\Omega} \partial_{2} \bar{b}^{1} \partial_{1} \widetilde{b}^{2} \partial_{2} \widetilde{u}^{2} d x d y+\int_{\Omega} \partial_{2} \widetilde{b}^{1} \partial_{1} b^{2} \partial_{2} u^{2}  d x d y \\
& \triangleq A_{211}+A_{212}.
\end{aligned}
$$
By anisotropic
H\"older inequality, interpolation \eqref{interpolation}, Poincar\'e inequality \eqref{Poincare} and Young's inequality,
\begin{equation}
\begin{aligned}
A_{211} & \leq\left\|\partial_{2} \bar{b}^{1}\right\|_{L^2_y}\left\|\partial_{1} \widetilde{b}^{2}\right\|_{L^\infty_yL^{2}_x}\left\|\partial_{2}\widetilde{u} ^{2}\right\|_{L^{2}_x L_{y}^{2}} \\
& \leq\left\|\partial_{2} \bar{b}^{1}\right\|_{L^{2}_y}\left\|\partial_{1} \widetilde{b}^{2}\right\|_{L^2}^{\frac{1}{2}}\left\|\partial_{1} \partial_{2} \widetilde{b}^{2}\right\|_{L^2}^{\frac{1}{2}}\left\|\partial_{2} \widetilde{u}^{2}\right\|_{L^2} \\
& \leq C\left\|\partial_{2} \bar{b}^{1}\right\|_{L^{2}_y} \left\|\partial_{1}  (u,b)\right\|_{L^2}\left\|\partial_{1} \partial_{2} (u,b)\right\|_{L^2} \\
& \leq C\left\|\partial_{2} b\right\|_{L^{2}}^{2}\left\|\partial_{1} (u,b)\right\|_{L^2}^{2}+\varepsilon\left\|\partial_{1} \partial_{2} (u,b)\right\|_{L^{2}}^{2}.
\end{aligned}
\end{equation}
Along the same way, we can bound $A_{112}$ by

\begin{align}
A_{212} &=\int_{\Omega} \partial_{2} \widetilde{b}^{1} \partial_{1} b^{2} \partial_{2} u^{2} d x d y \notag\\ & \leq\left\|\partial_{2} \widetilde{b}^{1 }\right\|_{L^\infty_xL^2_y}\left\|\partial_{1} b^{2}\right\|_{L^\infty_yL^2_x}\left\|\partial_{2} u^{2}\right\|_{L^2} \notag\\ & \leq C\left\|\partial_{2} \widetilde{b}^{1}\right\|_{L^2}^{\frac{1}{2}}\left\|\partial_{1} \partial_{2} \widetilde{b}^{1}\right\|_{L^2}^{\frac{1}{2}}\left\|\partial_{1} b^{2}\right\|_{L^2}^{\frac{1}{2}}\left\|\partial_{1} \partial_{2} b^{2}\right\|_{L^2}^{\frac{1}{2}}\left\|\partial_{2} u^{2}\right\|_{L^2} \notag\\ & \leq C\left\|\partial_{2} (u,b)\right\|_{L^2}^{2}\left\|\partial_{1} (u,b)\right\|_{L^2}^{2}+\varepsilon\left\|\partial_{1} \partial_{2} b\right\|_{L^2}^{2}.
\end{align} 
For $A_{22}$,
according to the divergence-free condition of $b$ and similar as $A_{21}$, we can bound it by
\begin{equation*} 
\begin{split}
A_{22}  \leq C\left\|\partial_{2} (u,b)\right\|_{L^2}^{2}\left\|\partial_{1} (u,b)\right\|_{L^2}^{2}+\varepsilon\left\|\partial_{1} \partial_{2} (u,b)\right\|_{L^2}^{2}.
\end{split} 
\end{equation*}
Thus we conclude that, the estimate for $A_{2}$ is  
\begin{equation}
\label{A-2} 
\begin{split}
A_{2}  \leq C\left\|\partial_{2} (u,b)\right\|_{L^2}^{2}\left\|\partial_{1} (u,b)\right\|_{L^2}^{2}+\varepsilon\left\|\partial_{1} \partial_{2} (u,b)\right\|_{L^2}^{2}.
\end{split} 
\end{equation}
Next we estimate $A_3$, 
firstly we decompose it as
$$
\begin{aligned}
A_{3} &=-\int_{\Omega} \partial_{2} u \cdot \nabla b \cdot \partial_{2} b d x d y \\
&=-\int_{\Omega} \partial_{2} u^{1} \partial_{1} b^{1} \partial_{2} b^{1} d x d y-\int_{\Omega} \partial_{2} u^{1} \partial_{1} b^{2} \partial_{2} b^{2} d x d y \\
&\quad-\int_{\Omega} \partial_{2} u^{2} \pa_{2} b^{1} \partial_{2} b^{1} d x d y-\int_{\Omega} \partial_{2} u^{2} \partial_{2} b^{2} \partial_{2} b^{2} d x d y \\
&\triangleq A_{31}+A_{32}+A_{33}+A_{34}.
\end{aligned}
$$
Similar as $A_2$, the terms $A_{32}$ and $A_{34}$ can be bounded by
\begin{align*}
A_{32},A_{34}  \leq C\left\|\partial_{2} (u,b)\right\|_{L^2}^{2}\left\|\partial_{1} (u,b)\right\|_{L^2}^{2}+\varepsilon\left\|\partial_{1} \partial_{2} (u,b)\right\|_{L^2}^{2}.
\end{align*} 
For the term $A_{31}$, we can decompose it as 
\begin{align*}
A_{31} &=-\int_{\Omega} \pa_{2} u^{1} \partial_{1} b^{1} \partial_{2} b^{1} d x d y \\
&=-\int_{\Omega} \partial_{2} {u}^{1} \partial_{1} b^{1} \partial_{2} \bar{b}^{1} d x d y-\int_{\Omega} \partial_{2} {u}^{1} \partial_{1} b^{1} \partial_{2} \widetilde{b}^{1} d x d y \\
&=-\int_{\Omega} \partial_{2} \bar{u}^{1} \partial_{1} \widetilde{b}^{1} \partial_{2} \bar{b}^{1} d x d y-\int_{\Omega} \partial_{2} \widetilde{u}^{1} \partial_{1} \widetilde{b}^{1} \partial_{2} \bar{b}^{1} d x d y\\
&\quad-\int_{\Omega} \partial_{2} {u}^{1} \partial_{1} b^{1} \partial_{2} \widetilde{b}^{1}  d x d y \\
& \triangleq A_{311}+A_{312}+A_{313}.
\end{align*}
According to the decomposition \eqref{decomposition}, it is obvious that $A_{311}=0$. For $A_{312}$, we can bound it by anisotropic H\"older inequality, interpolation \eqref{interpolation} and Young's inequality,
\begin{align*}
A_{312} & \leq\left\|\partial_{2} \bar{b}^{1}\right\|_{L^2_y}\left\|\partial_{2} \widetilde{u}^{1}\right\|_{L^2_xL^{2}_y}\left\|\partial_{1}\widetilde{b} ^{1}\right\|_{L^{\infty}_y L_{x}^{2}} \\
& \leq\left\|\partial_{2} \bar{b}^{1}\right\|_{L^{2}_y}\left\|\partial_{2} \widetilde{u}^{1}\right\|_{L^2}\left\|\partial_{1} \widetilde{b}^{1}\right\|_{L^2}^{\frac{1}{2}}\left\|\partial_{1} \partial_{2} \widetilde{b}^{1}\right\|_{2}^{\frac{1}{2}} \\
& \leq C\left\|\partial_{2} {b}^{1}\right\|_{L^{2}} \left\|\partial_{1}  (u,b)\right\|_{L^2}^{\frac12}\left\|\partial_{1} \partial_{2} (u,b)\right\|_{L^2}^{\frac32} \\
& \leq C\left\|\partial_{2} b\right\|_{L^{2}}^{4}\left\|\partial_{1} (u,b)\right\|_{L^2}^{2}+\varepsilon\left\|\partial_{1} \partial_{2} (u,b)\right\|_{L^{2}}^{2},
\end{align*}
where we have used Poincar\'e inequality \eqref{Poincare} in the third step.\\
Similarly,
\begin{equation} 
\begin{split}
A_{313} &=\int_{\Omega} \partial_{2} u^{1} \partial_{1} b^{1} \partial_{2} \widetilde{b}^{1} d x d y \\ & \leq C\left\|\partial_{2} u^{1}\right\|_{L^2}\left\|\partial_{1} {b}^{1 }\right\|_{L^\infty_yL^2_x}\left\|\partial_{2} \widetilde{b}^{1}\right\|_{L^\infty_xL^2_y} \\ & \leq C\left\|\partial_{2} u^{1}\right\|_{L^2}\left\|\partial_{1} {b}^{1}\right\|_{L^2}^{\frac{1}{2}}\left\|\partial_{1} \partial_{2} {b}^{1}\right\|_{L^2}^{\frac{1}{2}}\left\|\partial_{2} \widetilde{b}^{1}\right\|_{L^2}^{\frac{1}{2}}\left\|\partial_{1} \partial_{2} \widetilde{b}^{1}\right\|_{L^2}^{\frac{1}{2}} \\ 
& \leq C\left\|\partial_{2} {u}^{1}\right\|_{L^{2}} \left\|\partial_{1}  (u,b)\right\|_{L^2}^{\frac12}\left\|\partial_{1} \partial_{2} (u,b)\right\|_{L^2}^{\frac32} \\
& \leq C\left\|\partial_{2} u\right\|_{L^2}^{4}\left\|\partial_{1}  (u,b)\right\|_{L^2}^{2}+\varepsilon\left\|\partial_{1} \partial_{2} u\right\|_{L^2}^{2}.
\end{split} 
\end{equation}
By the divergence-free condition of $u$, 
\begin{equation*}
A_{33}=-\int_{\Omega}\pa_2u^2\pa_2b^1\pa_2b^1 d x d y=\int_{\Omega}\pa_1u^1\pa_2b^1\pa_2b^1 d x d y.
\end{equation*}
It follows by the same method as in estimate of $A_{31}$, we can obtain
\begin{equation*} 
\begin{split}
A_{33} \leq C(\left\|\partial_{2} (u,b)\right\|_{L^2}^{2}+\left\|\partial_{2} (u,b)\right\|_{L^2}^{4}) \left\|\partial_{1} (u,b)\right\|_{L^2}^{2}+\varepsilon\left\|\partial_{1} \partial_{2} (u,b)\right\|_{L^2}^{2}.
\end{split} 
\end{equation*}
Summing the estimates of $A_{31}-A_{34}$, we get the bound for $A_{3}$ that
\begin{equation}
\label{A-3} 
\begin{split}
A_{3} \leq C(\left\|\partial_{2} (u,b)\right\|_{L^2}^{2}+\left\|\partial_{2} (u,b)\right\|_{L^2}^{4}) \left\|\partial_{1} (u,b)\right\|_{L^2}^{2}+\varepsilon\left\|\partial_{1} \partial_{2} (u,b)\right\|_{L^2}^{2},
\end{split} 
\end{equation}
and the same conclusion can be drawn for $A_{4}$.\\
Adding the estimates \eqref{A-1}, \eqref{A-2} and \eqref{A-3} into \eqref{L2-pa_2ub},  choosing $\varepsilon=\frac{\min\{\nu,\eta\}}{16}$, one can deduce
\begin{equation}
\label{L2-pa_2ub}
\begin{split}
& \frac{d}{d t}\|\pa_2(u,b)(t)\|_{L^{2}}^{2}+\overline{C}\|\partial_{1}\pa_2 (u,b)\|_{2}^{2}\leq C(\left\|\partial_{2} (u,b)\right\|_{L^2}^{2}+\left\|\partial_{2} (u,b)\right\|_{L^2}^{4}) \left\|\partial_{1} (u,b)\right\|_{L^2}^{2},
\end{split}
\end{equation}
where $\overline{C}=\min\{\nu,\eta\}$. 
Denoting
\begin{equation}
\begin{split}
F(t)&=\sup_{0\leq\tau\leq t}\|(u,b)(t)\|_{L^2}^2+\sup_{0\leq\tau\leq t}\|\pa_2(u,b)(t)\|_{L^2}^2+
\overline{C}\int_0^t\|\pa_1(u,b)(\tau)\|_{L^2}^2 d\tau\\
&\quad+\overline{C}\int_0^t\|\pa_1\pa_2(u,b)(\tau)\|_{L^2}^2 d\tau.
\end{split}
\end{equation}
Integrating \eqref{L2-pa_2ub} from $0$ to $t$ with respect to time variable, combining with the inequality \eqref{L2-ub}, we conclude that
\begin{equation}
F(t)\leq F(0)+CF(t)^2+CF(t)^3.
\end{equation}
A bootstrapping argument implies that, there is $\delta>0$, such that, if $F(0)<\delta$, then 
$$F(t) \leq C_0 \delta$$
for a pure constant $C_0$ and for all $t > 0$. This completes the proof of this proposition.
\end{proof}

\end{section}

\begin{section}{Regularity estimate for $(u,b)$}
The goal of this section is to present the global regularity estimate for the solution $(u,b)$ to system \eqref{system-MHD}. The main result can be stated as follows.
\begin{prop}
\label{prop-Hs}
Assume the initial data $(u_0,b_0)\in H^s (s\geq 1)$ and satisfies $\nabla\cdot u_0=\nabla \cdot b_0=0$. Then there exist some sufficiently small positive constant $\delta$, such that if $(u_0,b_0)$ satisfies the smallness condition 
\eqref{smallness condition},
then the corresponding solution $(u,b)$ to system \eqref{system-MHD} satisfies 
\begin{equation*}
\begin{split}
u\in L^\infty(\R_+,H^s(\Omega)),\quad
\pa_1u\in L^2(\R_+,H^s(\Omega)),\\
b\in L^\infty(\R_+,H^s(\Omega)),\quad
\pa_1b\in L^2(\R_+,H^s(\Omega)).
\end{split}
\end{equation*}
\end{prop}

The proof of Proposition \ref{prop-Hs} deeply relays on the following lemma.
\begin{lemma}
\label{three-linear term}
Let $f, g, h$ be three smooth vector fields defined in $\Omega$ and $\nabla\cdot f=0$. Then we have the following estimate for the three-linear term
\begin{equation}
\begin{aligned}
&\quad-\int_{\Omega}\Delta_q(f\cdot \nabla g)\cdot\Delta_q h~dxdy\\
& \leq 
C2^{-2qs}b_q(\left\|\pa_1{f}\right\|_{L^2}\left\|\pa_1\pa_2 {f}\right\|_{L^2}+\left\|{f}\right\|_{L^2}^2\left\|\pa_1 {f}\right\|_{L^2}^2+\left\|{f}\right\|_{L^2}^2\left\|\pa_1\pa_2 {f}\right\|_{L^2}^2\\
&\quad+\left\|\pa_1{g}\right\|_{L^2}\left\|\pa_1\pa_2 {g}\right\|_{L^2}+\left\|{\pa_1\pa_2g}\right\|_{L^2}^2)\times(\left\| {f}\right\|_{H^s}^2+ \left\| g\right\|_{H^s}^2+\left\| h\right\|_{H^s}^2 ) \\
&\quad+2^{-2qs}b_q(\left\|{f}\right\|_{L^2}^{\frac12}\left\|\pa_2 {f}\right\|_{L^2}^{\frac12}+\left\|{\pa_2g}\right\|_{L^2})\times
(\left\| \pa_1{f}\right\|_{H^s}^2+\left\|\partial_{1} {g}\right\|_{H^s}^2+ \left\| \pa_1h\right\|_{H^s}^2)\\
&\quad+C\varepsilon_02^{-2qs}b_q(\left\| \pa_1{f}\right\|_{H^s}^2+\left\|\partial_{1} {g}\right\|_{H^s}^2+ \left\|\pa_1 h\right\|_{H^s}^2)-\int_{\Omega}{S}_{q}\widetilde{f}^2\partial_2\Delta_q g\cdot\Delta_q h~dxdy,
\end{aligned}
\end{equation} 
for any $\varepsilon_0>0$.
\end{lemma}
The proof of Lemma \ref{three-linear term} is based on Littlewood-Paley decomposition and the process is complicate, we left the details in the Appendix.

Now we give the proof of Proposition \ref{prop-Hs}.
\begin{proof}[Proof of Proposition \ref{prop-Hs}]
Applying $\Delta_q$ to the first and second equation of $\eqref{system-MHD}$, taking the $L^2$ inner product	of
the resulting equation with $(\Delta_q u, \Delta_q b)$, after integrating by part we obtain
\begin{equation*}
\begin{split}
&\quad\frac{1}{2} \frac{d}{d t}\left(\|\Delta_qu(t)\|_{L^{2}}^{2}+\|\Delta_qb(t)\|_{L^{2}}^{2}\right)+\nu\|\partial_{1}\Delta_q u\|_{2}^{2}+\eta\|\partial_{1}\Delta_q b\|_{2}^{2}\\
&=-\int_{\Omega}\Delta_q(u\cdot\nabla u)\cdot\Delta_q u~dxdy+\int_{\Omega}\Delta_q(b\cdot\nabla b)\cdot\Delta_q u~dxdy\\
&\quad-\int_{\Omega}\Delta_q(u\cdot\nabla b)\cdot\Delta_q b~dxdy+\int_{\Omega}\Delta_q(b\cdot\nabla u)\cdot\Delta_q b~dxdy.
\end{split}
\end{equation*}
Applying Lemma \ref{three-linear term},
\begin{align}
\label{Hs-ub}
&\quad\frac{1}{2} \frac{d}{d t}\left(\|\Delta_qu(t)\|_{L^{2}}^{2}+\|\Delta_qb(t)\|_{L^{2}}^{2}\right)+\nu\|\partial_{1}\Delta_q u\|_{2}^{2}+\eta\|\partial_{1}\Delta_q b\|_{2}^{2}\notag\\
& \leq 
C2^{-2qs}b_q(\left\|\pa_1{u}\right\|_{L^2}\left\|\pa_1\pa_2 {u}\right\|_{L^2}+\left\|{u}\right\|_{L^2}^2\left\|\pa_1 {u}\right\|_{L^2}^2+\left\|{u}\right\|_{L^2}^2\left\|\pa_1\pa_2 {u}\right\|_{L^2}^2+\left\|{\pa_1\pa_2u}\right\|_{L^2}^2\notag\\
&\quad+\left\|\pa_1{b}\right\|_{L^2}\left\|\pa_1\pa_2 {b}\right\|_{L^2}+\left\|{b}\right\|_{L^2}^2\left\|\pa_1 {b}\right\|_{L^2}^2+\left\|{b}\right\|_{L^2}^2\left\|\pa_1\pa_2 {b}\right\|_{L^2}^2+\left\|{\pa_1\pa_2b}\right\|_{L^2}^2)\notag\\
&\quad\times(\left\| {u}\right\|_{H^s}^2+ \left\| b\right\|_{H^s}^2 )+C2^{-2qs}b_q(\left\|{u}\right\|_{L^2}^{\frac12}\left\|\pa_2 {u}\right\|_{L^2}^{\frac12}+\left\|{\pa_2u}\right\|_{L^2}+\left\|{b}\right\|_{L^2}^{\frac12}\left\|\pa_2 {b}\right\|_{L^2}^{\frac12} \notag\\
&\quad+\left\|{\pa_2b}\right\|_{L^2})\times
(\left\| \pa_1{u}\right\|_{H^s}^2+\left\|\partial_{1} {b}\right\|_{H^s}^2)+C\varepsilon_02^{-2qs}b_q(\left\| \pa_1{u}\right\|_{H^s}^2+\left\|\partial_{1} {b}\right\|_{H^s}^2)\notag\\
& \quad-\int_{\Omega}{S}_{q}\widetilde{u}^2\partial_2\Delta_q u\cdot\Delta_q u~dxdy+\int_{\Omega}{S}_{q}\widetilde{b}^2\partial_2\Delta_q b\cdot\Delta_q u~dxdy\notag\\
&\quad-\int_{\Omega}{S}_{q}\widetilde{u}^2\partial_2\Delta_q b\cdot\Delta_q b~dxdy
+ \int_{\Omega}{S}_{q}\widetilde{b}^2\partial_2\Delta_q u\cdot\Delta_q b~dxdy.
\end{align}
Now we need to deal with the last four terms in the right-hand side.
After integrating by part and combining with the divergence-free condition of $\widetilde{u}$, we can write 
\begin{equation*}
-\int_{\Omega}{S}_{q}\widetilde{u}^2\partial_2\Delta_q u\Delta_q u~dxdy=\frac12\int_{\Omega}\partial_2{S}_{q}\widetilde{u}^2\Delta_q u\cdot\Delta_q u~dxdy=-\frac12\int_{\Omega}\partial_1{S}_{q}\widetilde{u}^1\Delta_q u\cdot\Delta_q u~dxdy.
\end{equation*}
Then according to the decomposition \eqref{decomposition}, we can divide it into the following four parts,
\begin{align*}
&\quad-\frac12\int_{\Omega}\partial_1{S}_{q}\widetilde{u}^1\Delta_q u\cdot\Delta_q u~dxdy\\
&=-\frac12\int_{\Omega}\partial_1{S}_{q}\widetilde{u}^1\Delta_q \bar{u}\cdot\Delta_q \bar{u}~dxdy-\frac12\int_{\Omega}\partial_1{S}_{q}\widetilde{u}^1\Delta_q \bar{u}\cdot\Delta_q \widetilde{u}~dxdy\\
&\quad-\frac12\int_{\Omega}\partial_1{S}_{q}\widetilde{u}^1\Delta_q \widetilde{u}\cdot\Delta_q \bar{u}~dxdy-\frac12\int_{\Omega}\partial_1{S}_{q}\widetilde{u}^1\Delta_q \widetilde{u}\cdot\Delta_q \widetilde{u}~dxdy\\
&\triangleq D_{1}+D_{2}+D_{3}+D_{4}.
\end{align*}
For $D_1$, noticing that $\bar{u}$ independent of $x$. So it is a simple matter to check that $D_1=0$. 
Then we estimate $D_2$. We first apply anisotropic H\"older inequality, 
\begin{equation*}
\begin{split}
D_2&=-\frac12\int_{\Omega}\partial_1{S}_{q}\widetilde{u}^1\Delta_q \bar{u}\cdot\Delta_q \widetilde{u}~dxdy\leq C\left\|\pa_1S_q\widetilde{u}^1\right\|_{L^\infty_yL^2_x}\left\|\Delta_q\bar{u}\right\|_{L^2_y}\left\|\Delta_q\widetilde{u}\right\|_{L^2_yL^2_x}.
\end{split}
\end{equation*}
Then by the interpolation \eqref{interpolation}
and Poincar\'e inequality \eqref{Poincare}, 
\begin{equation*}
\begin{split}
D_2&\leq C\left\|\pa_1S_q\widetilde{u}^1\right\|_{L^2}^{\frac12}\left\|\pa_1\pa_2S_q\widetilde{u}^1\right\|_{L^2}^{\frac12}\left\|\Delta_q\bar{u}\right\|_{L^2_y}\left\|\pa_1\Delta_q\widetilde{u}\right\|_{L^2_yL^2_x}\\
&\leq C2^{-2qs}b_q\|\pa_1u\|_{L^2}^{\frac12}\|\pa_1\pa_2u\|_{L^2}^{\frac12}\|u\|_{H^s}\|\pa_1u\|_{H^s}.
\end{split}
\end{equation*}
The same conclusion can also be drawn for $D_3$ and $D_4$ along the same way. Thus we deduce
\begin{equation*}
\begin{split}
-\frac12\int_{\Omega}\partial_1{S}_{q}\widetilde{u}^1\Delta_q u\cdot\Delta_q u~dxdy\leq C2^{-2qs}b_q\|\pa_1u\|_{L^2}^{\frac12}\|\pa_1\pa_2u\|_{L^2}^{\frac12}\|u\|_{H^s}\|\pa_1u\|_{H^s}.
\end{split}
\end{equation*}
Similarly, we can also obtain
\begin{equation*}
\begin{split}
-\int_{\Omega}{S}_{q}\widetilde{u}^2\partial_2\Delta_q b\cdot\Delta_q b~dxdy\leq C2^{-2qs}b_q\|\pa_1u\|_{L^2}^{\frac12}\|\pa_1\pa_2u\|_{L^2}^{\frac12}\|b\|_{H^s}\|\pa_1b\|_{H^s}.
\end{split}
\end{equation*}
Applying the integration by parts and according to the divergence-free property of $\widetilde{b}$,
\begin{equation*}
\begin{split}
&\quad\int_{\Omega}{S}_{q}\widetilde{b}^2\partial_2\Delta_q b\cdot\Delta_q u~dxdy+\int_{\Omega}{S}_{q}\widetilde{b}^2\partial_2\Delta_q u\cdot\Delta_q b~dxdy\\
&=-\int_{\Omega}\pa_2{S}_{q}\widetilde{b}^2\Delta_q b\cdot\Delta_q u~dxdy=\int_{\Omega}\pa_1{S}_{q}\widetilde{b}^1\Delta_q b\cdot\Delta_q u~dxdy.
\end{split}
\end{equation*}
Then this term can be handled following the same method that
\begin{equation*}
\begin{split}
\int_{\Omega}\pa_1{S}_{q}\widetilde{b}^1\Delta_q b\cdot\Delta_q u~dxdy\leq C2^{-2qs}b_q\|\pa_1b\|_{L^2}^{\frac12}\|\pa_1\pa_2b\|_{L^2}^{\frac12}(\|u\|_{H^s}\|\pa_1b\|_{H^s}+\|b\|_{H^s}\|\pa_1u\|_{H^s}).
\end{split}
\end{equation*}
Inserting these estimate into \eqref{Hs-ub}, one can deduce
\begin{equation}
\label{Hs-ub-1}
\begin{split}
&\quad\frac{1}{2} \frac{d}{d t}\left(\|\Delta_qu(t)\|_{L^{2}}^{2}+\|\Delta_qb(t)\|_{L^{2}}^{2}\right)+\nu\|\partial_{1}\Delta_q u\|_{2}^{2}+\eta\|\partial_{1}\Delta_q b\|_{2}^{2}\\
& \leq 
C2^{-2qs}b_q(\left\|\pa_1{u}\right\|_{L^2}\left\|\pa_1\pa_2 {u}\right\|_{L^2}+\left\|{u}\right\|_{L^2}^2\left\|\pa_1 {u}\right\|_{L^2}^2+\left\|{u}\right\|_{L^2}^2\left\|\pa_1\pa_2 {u}\right\|_{L^2}^2+\left\|{\pa_1\pa_2u}\right\|_{L^2}^2\\
&\quad+\left\|\pa_1{b}\right\|_{L^2}\left\|\pa_1\pa_2 {b}\right\|_{L^2}+\left\|{b}\right\|_{L^2}^2\left\|\pa_1 {b}\right\|_{L^2}^2+\left\|{b}\right\|_{L^2}^2\left\|\pa_1\pa_2 {b}\right\|_{L^2}^2+\left\|{\pa_1\pa_2b}\right\|_{L^2}^2)\\
&\quad\times(\left\| {u}\right\|_{H^s}^2+ \left\| b\right\|_{H^s}^2 )+C2^{-2qs}b_q(\left\|{u}\right\|_{L^2}^{\frac12}\left\|\pa_2 {u}\right\|_{L^2}^{\frac12}+\left\|{\pa_2u}\right\|_{L^2}+\left\|{b}\right\|_{L^2}^{\frac12}\left\|\pa_2 {b}\right\|_{L^2}^{\frac12} \\
&\quad+\left\|{\pa_2b}\right\|_{L^2})\times
(\left\| \pa_1{u}\right\|_{H^s}^2+\left\|\partial_{1} {b}\right\|_{H^s}^2)+C\varepsilon_02^{-2qs}b_q(\left\| \pa_1{u}\right\|_{H^s}^2+\left\|\partial_{1} {b}\right\|_{H^s}^2).
\end{split}
\end{equation}
Choosing $\delta$ and $\varepsilon_0$ small (for example $\sqrt{\delta}=\varepsilon_0=\frac{\min\{\nu,\eta\}}{4C(1+C_0)\sum_qb_q}$), multiplying \eqref{Hs-ub-1} by $2^{2qs}$
and taking summation in $q$, combining with the result in Proposition \ref{prop-2.1}, we can deduce
\begin{equation}
\label{Hs-ub-2}
\begin{split}
&\quad\frac{1}{2} \frac{d}{d t}\left(\|u(t)\|_{H^s}^{2}+\|b(t)\|_{H^s}^{2}\right)+\frac{\overline{C}}{2}(\|\partial_{1} u\|_{H^s}^{2}+\|\partial_{1} b\|_{H^s}^{2})\\
& \leq 
C(\left\|\pa_1{u}\right\|_{L^2}\left\|\pa_1\pa_2 {u}\right\|_{L^2}+\left\|{u}\right\|_{L^2}^2\left\|\pa_1 {u}\right\|_{L^2}^2+\left\|{u}\right\|_{L^2}^2\left\|\pa_1\pa_2 {u}\right\|_{L^2}^2+\left\|{\pa_1\pa_2u}\right\|_{L^2}^2\\
&\quad+\left\|\pa_1{b}\right\|_{L^2}\left\|\pa_1\pa_2 {b}\right\|_{L^2}+\left\|{b}\right\|_{L^2}^2\left\|\pa_1 {b}\right\|_{L^2}^2+\left\|{b}\right\|_{L^2}^2\left\|\pa_1\pa_2 {b}\right\|_{L^2}^2+\left\|{\pa_1\pa_2b}\right\|_{L^2}^2)\\
&\quad\times(\left\| {u}\right\|_{H^s}^2+ \left\| b\right\|_{H^s}^2 ).
\end{split}
\end{equation}
Applying Gr\"onwall's Lemma and combining with the uniform global bound for $(u,b)$ in Proposition \eqref{prop-2.1}, we can obtain the desired result, which completes the proof of this proposition.
\end{proof} 

\end{section}

\section{proof of theroem \ref{thm1}}
This section is devoted to proving Theorem \ref{thm1}. To prove the existence, we need to add an artificial viscosity $-\epsilon\Delta(u,b)$ on system \eqref{system-MHD} and by regularizing the initial data. For this fully parabolic system with smooth initial data, we have a unique global solution $(u_{\epsilon},b_{\epsilon})$ by the classical result on MHD system. It is easy to see that $(u_{\epsilon},b_{\epsilon})$ obeys the a priori bounds in Proposition \ref{prop-2.1}, Proposition \ref{prop-Hs} and uniformly in $\epsilon$. The solution $(u,b)$ of system \ref{system-MHD} is obtained as a limit of $(u_{\epsilon},b_{\epsilon})$ and obeys the bounds in Proposition \ref{prop-2.1} and Proposition \ref{prop-Hs}. Because these processes are classical, we will not give all the details about the construction of the solution, and let them for the readers.

Then we prove the uniqueness. Assume $(u_1,b_1,p_1)$ and $(u_2,b_2,p_2)$ are two solutions of system \eqref{system-MHD} with the same initial data. Denoting 
$\delta u\triangleq u_2-u_1$, $\delta b\triangleq b_2-b_1$,
$\delta p\triangleq p_2-p_1$, then we can obtain $(\delta u,\delta b,\delta p)$ satisfies:
\begin{equation}
\label{system-difference}
\left\{\begin{array}{l}
{\partial_{t} \delta u+u_2\cdot \nabla \delta u+\delta u\cdot\nabla u_1-\nu\partial_1^{2} \delta u+\nabla \delta p=b_2 \cdot \nabla \delta b}+\delta b\cdot\nabla b_1,  \\
{\partial_{t} \delta b+u_2\cdot \nabla \delta b+\delta u\cdot\nabla b_1-\eta\partial_1^{2} \delta b=b_2 \cdot \nabla \delta u}+\delta b\cdot\nabla u_1, \\
\nabla\cdot\delta u=\nabla\cdot\delta b=0,\\
\delta u(x, y, 0)=0,~~\delta b(x, y, 0)=0.
\end{array}\right.
\end{equation}
Standard $L^2$ estimate yields that
\begin{equation}
\label{uniqueness-1}
\begin{split}
&\quad\frac{1}{2}\frac{d}{dt}(\|\delta u(t)\|_{L^2}^2+\|\delta b(t)\|_{L^2}^2)+\nu\|\pa_1\delta u\|_{L^2}^2+\eta\|\pa_1\delta b\|_{L^2}^2\\
&=-\int_{\Omega}\delta u\cdot\nabla u_1\cdot\delta u~dxdy+\int_{\Omega}\delta b\cdot\nabla u_1\cdot\delta u~dxdy\\
&\quad-\int_{\Omega}\delta u\cdot\nabla b_1\cdot\delta b~dxdy+\int_{\Omega}\delta b\cdot\nabla u_1\cdot\delta b~dxdy\\
&\triangleq K_1+K_2+K_3+K_4.
\end{split}
\end{equation}
By anisotropic H\"older inequality, interpolation inequality and Young's inequality, 
\begin{equation*}
\begin{split}
K_1&=-\int_{\Omega}\delta u\cdot\nabla u_1\cdot\delta u~dxdy\\
&\leq\|\delta u\|_{L^2}\|\nabla u_1\|_{L^\infty_yL^2_x}\|\delta u\|_{L^\infty_xL^2_y}\\
&\leq C\|\delta u\|_{L^2}\|\nabla u_1\|_{L^2}^{\frac12}\|\pa_2\nabla u_1\|_{L^2}^{\frac12}(\|\delta u\|_{L^2}^{\frac12}\|\pa_1\delta u\|_{L^2}^{\frac12}+\|\delta u\|_{L^2})\\
&\leq C\|\delta u\|_{L^2}^2\left(\|\nabla u_1\|_{L^2}^{\frac12}\|\pa_2\nabla u_1\|_{L^2}^{\frac12}+\|\nabla u_1\|_{L^2}^{\frac23}\|\pa_2\nabla u_1\|_{L^2}^{\frac23}\right)+\frac{\nu}{4}\|\pa_1\delta u\|_{L^2}^2.
\end{split}
\end{equation*}
Along the same way, one can deduce
\begin{equation*}
\begin{split}
K_2
&\leq C\|\delta u\|_{L^2}\|\delta b\|_{L^2}\left(\|\nabla u_1\|_{L^2}^{\frac12}\|\pa_2\nabla u_1\|_{L^2}^{\frac12}+\|\nabla u_1\|_{L^2}^{\frac23}\|\pa_2\nabla u_1\|_{L^2}^{\frac23}\right)+\frac{\nu}{4}\|\pa_1\delta u\|_{L^2}^2,
\end{split}
\end{equation*}
\begin{equation*}
\begin{split}
K_3
&\leq C\|\delta u\|_{L^2}\|\delta b\|_{L^2}\left(\|\nabla b_1\|_{L^2}^{\frac12}\|\pa_2\nabla b_1\|_{L^2}^{\frac12}+\|\nabla b_1\|_{L^2}^{\frac23}\|\pa_2\nabla b_1\|_{L^2}^{\frac23}\right)+\frac{\eta}{4}\|\pa_1\delta b\|_{L^2}^2,
\end{split}
\end{equation*}
and
\begin{equation*}
\begin{split}
K_4
&\leq C\|\delta b\|_{L^2}^2\left(\|\nabla u_1\|_{L^2}^{\frac12}\|\pa_2\nabla u_1\|_{L^2}^{\frac12}+\|\nabla u_1\|_{L^2}^{\frac23}\|\pa_2\nabla u_1\|_{L^2}^{\frac23}\right)+\frac{\eta}{4}\|\pa_1\delta b\|_{L^2}^2.
\end{split}
\end{equation*}
Inserting the estimates of $K_1-K_4$ into inequality \eqref{uniqueness-1}, combining with the bounds in Proposition \ref{prop-Hs} we can conclude the uniqueness result.

\begin{section}{proof of theroem \ref{thm2}}
	
This section focuses on the proof of Theorem \ref{thm2}. 
According to Lemma \ref{properties1} and Lemma \ref{properties2}, it is a simple matter to check that $(\bar{u},\bar{b})$ satisfies the following system,
	\begin{equation}
	\label{bar-MHD}
	\left\{\begin{array}{l}
	{\partial_{t} \bar{u}+\overline{u\cdot \nabla \widetilde{u}}+\left(\begin{array}{c}
		{0} \\
		{\partial_{2} \bar{p}}
		\end{array}\right)
		=\overline{b \cdot \nabla \widetilde{b}}}, \\
	{\partial_t \bar{b}+\overline{u\cdot \nabla \widetilde{b}}=\overline{b\cdot \nabla \widetilde{u}}}.
	\end{array}\right.
	\end{equation}
Taking difference between \eqref{system-MHD} and \eqref{bar-MHD}, also by the properties in Lemma \ref{properties1} and Lemma \ref{properties2}, we can deduce
	\begin{equation}
	\label{tilde-MHD}
	\left\{\begin{array}{l}
	{\partial_{t} \widetilde{u}+\widetilde{u\cdot \nabla \widetilde{u}}+u^2\pa_2\bar{u}-\nu\pa_1^2\widetilde{u}+\nabla\widetilde{p}
		=\widetilde{b \cdot \nabla \widetilde{b}}+b^2\pa_2\overline{b},}\\
	{\partial_{t} \widetilde{b}+\widetilde{u\cdot \nabla \widetilde{b}}+u^2\pa_2\bar{b}-\eta\pa_1^2\widetilde{b}
		=\widetilde{b \cdot \nabla \widetilde{u}}+b^2\pa_2\overline{u}.}
	\end{array}\right.
	\end{equation}
Basic $L^2$ estimate yields that 
	\begin{equation}
	\begin{split}
	&\quad\frac{1}{2} \frac{d}{d t}(\|\widetilde{u}(t)\|_{L^{2}}^{2}+\| \widetilde{b}(t)\|_{L^2}^2)+\nu\|\partial_{1} \widetilde{u}\|_{2}^{2}+\eta\|\partial_{1} \widetilde{b}\|_{2}^{2}\\
	&=-\int_{\Omega}\widetilde{u\cdot \nabla \widetilde{u}}\cdot\widetilde{u}~dxdy+\int_{\Omega}\widetilde{b\cdot \nabla \widetilde{b}}\cdot\widetilde{u}~dxdy-\int_{\Omega}\widetilde{u\cdot \nabla \widetilde{b}}\cdot\widetilde{b}~dxdy+\int_{\Omega}\widetilde{b\cdot \nabla \widetilde{u}}\cdot\widetilde{b}~dxdy\\
	&\quad- \int_{\Omega}u^2\pa_2 \bar{u}\cdot\widetilde{u}~dxdy+ \int_{\Omega}b^2\pa_2 \bar{b}\cdot\widetilde{u}~dxdy- \int_{\Omega}u^2\pa_2 \bar{b}\cdot\widetilde{b}~dxdy+ \int_{\Omega}b^2\pa_2 \bar{u}\cdot\widetilde{b}~dxdy\\
	&\triangleq M_{11}+M_{12}+M_{13}+M_{14}+M_{21}+M_{22}+M_{23}+M_{24}.
	\end{split}
	\end{equation}
For $M_{11}$, according to the divergence-free condition of $u$, 
	\begin{equation*}
	\begin{split}
	M_{11}&=-\int_{\Omega}\widetilde{u\cdot \nabla \widetilde{u}}\cdot\widetilde{u}~dxdy\\
	&=-\int_{\Omega}{u\cdot \nabla \widetilde{u}}\cdot\widetilde{u}~dxdy+\int_{\Omega}\overline{u\cdot \nabla \widetilde{u}}\cdot\widetilde{u}~dxdy\\
	&=0.
	\end{split}
	\end{equation*}
Similarly, we can prove 
	\begin{equation*}
	M_{13}=0\quad \text{and} \quad M_{12}+M_{14}=0.
	\end{equation*}
For $M_{21}$, according to Property \eqref{u2=0},
	\begin{equation*}
	M_{21}=- \int_{\Omega}u^2\pa_2 \bar{u}\cdot\widetilde{u}~dxdy=- \int_{\Omega}\widetilde{u}^2\pa_2 \bar{u}\cdot\widetilde{u}~dxdy.
	\end{equation*}
By H\"older inequality, interpolation \eqref{interpolation} and Poincar\'e inequality \eqref{Poincare},  
	\begin{equation*}
	\begin{split}
	M_{21}&\leq C\|\pa_2 \bar{u}\|_{L^2_y}\|\widetilde{u}^2\|_{L^\infty_yL^2_x}\|\widetilde{u}\|_{L^2_yL^2_x}\\
	&\leq C\|\pa_2 \bar{u}\|_{L^2_y}\|\widetilde{u}^2\|_{L^2}^\frac12\|\pa_2\widetilde{u}^2\|_{L^2}^\frac12\|\widetilde{u}\|_{L^2}\\
	&\leq C\|\pa_2 u\|_{L^2}\|\pa_1\widetilde{u}\|_{L^2}^2.
	\end{split}
	\end{equation*}
Along the same way, we can obtain
	\begin{equation*}
	M_{21},M_{23},M_{24}\leq C(\|\pa_2 u\|_{L^2}+\|\pa_2 b\|_{L^2})(\|\pa_1\widetilde{u}\|_{L^2}^2+\|\pa_1\widetilde{b}\|_{L^2}^2).
	\end{equation*}
Then combining with the result in Proposition \ref{prop-2.1}, for small $\delta$,
	\begin{equation}
	\begin{split}
	\frac{d}{d t}(\|\widetilde{u}(t)\|_{L^{2}}^{2}+\| \widetilde{b}(t)\|_{L^2}^2)+\overline{C}(\|\partial_{1} \widetilde{u}\|_{2}^{2}+\|\partial_{1} \widetilde{b}\|_{2}^{2})\leq0.
	\end{split}
	\end{equation}
Using Poincar\'e inequality \eqref{Poincare},
	\begin{equation}
	\begin{split}
	\frac{d}{d t}(\|\widetilde{u}(t)\|_{L^{2}}^{2}+\| \widetilde{b}(t)\|_{L^2}^2)+\overline{C}_0(\| \widetilde{u}\|_{2}^{2}+\| \widetilde{b}\|_{2}^{2})\leq0.
	\end{split}
	\end{equation}
Then Gr\"onwall's Lemma yields that
	\begin{equation}
	\begin{split}
	\|\widetilde{u}(t)\|_{L^{2}}^{2}+\| \widetilde{b}(t)\|_{L^2}^2\leq e^{-\overline{C}_0t}(\|\widetilde{u}(0)\|_{L^{2}}^{2}+\| \widetilde{b}(0)\|_{L^2}^2)\leq Ce^{-\overline{C}_0t}(\|u_0\|_{L^{2}}^{2}+\| b_0\|_{L^2}^2),
	\end{split}
	\end{equation}
which completes the proof of Theorem \ref{thm2}.
	
\end{section}

\begin{section}{Applications to the tropical climatic model}
	
	The goal of this section is to present the proof of Theorem \ref{thm1-TC}. For the existence part, similar as the proof of Theorem \ref{thm1}, it basically follows from the a priori estimates for
	smooth enough solutions of \eqref{system-tropical}. We shall only outline the main steps in the derivation of the a priori estimates. The first lemma gives the globla $L^2$ bound for $(u,v)$ and $(\pa_2u,\pa_2v)$.

	\begin{lemma}
		\label{lem1-TC}
		Assume $(u_0,b_0)$ satisfies the assumptions in Theorem \ref{thm1-TC}. Then the smooth solution $(u,b)$ of system \eqref{TC} with the initial data $(u_0,b_0)$ satisfies
		\begin{equation*}
		\begin{split}
		u,&b\in L^\infty(\R_+;L^2(\Omega)), \quad \pa_1u,\pa_1b\in L^2(\R_+;L^2(\Omega)),\\
		\pa_2u,&\pa_2b\in L^\infty(\R_+;L^2(\Omega)), \quad \pa_1\pa_2u,\pa_1\pa_2b\in L^2(\R_+;L^2(\Omega)).
		\end{split}
		\end{equation*}
	\end{lemma}
	
	Testing the equations \eqref{system-tropical}$_1$ and \eqref{system-tropical}$_2$ by $u$ and $v$, respectively, and adding them up, we can find
	\begin{equation}
	\label{L2-uv}
	\frac12\frac{d}{dt}(\|u(t)\|_{L^2}^2+\|v(t)\|_{L^2}^2)+\nu\|\pa_1u\|_{L^2}^2+\eta\|\pa_1v\|_{L^2}^2=0,
	\end{equation}
	where we have used the facts
	\begin{equation*}
	\int_{\Omega}(u\cdot\nabla u)\cdot u~dxdy=\int_{\Omega}(u\cdot\nabla v)\cdot v~dxdy=0,
	\end{equation*}
	and
	\begin{equation*}
	\int_{\Omega}(u\cdot\nabla u)\cdot u~dxdy=\int_{\Omega}(u\cdot\nabla v)\cdot v~dxdy=0.
	\end{equation*}
	Then we derive the $L^2$ estimate for $(\pa_2u,\pa_2v)$. Testing $(\pa_2u,\pa_2v)$ to the equations of $\pa_2u$ and $\pa_2v$, we have 
	\begin{equation}
	\label{L2-pa_2uv-1}
	\begin{split}
	&\quad\frac{1}{2} \frac{d}{d t}\left(\|\pa_2u(t)\|_{L^{2}}^{2}+\|\pa_2v(t)\|_{L^{2}}^{2}\right)+\nu\|\partial_{1}\pa_2 u\|_{2}^{2}+\eta\|\partial_{1}\pa_2 v\|_{2}^{2}\\
	&=-\int_{\Omega}\pa_2u\cdot\nabla u\cdot\pa_2 u~dxdy-\int_{\Omega}\pa_2b\cdot\nabla b\cdot\pa_2 u~dxdy-\int_{\Omega}b\cdot\nabla\pa_2b\cdot\pa_2u~dxdy\\
	&\quad-\int_{\Omega}\pa_2u\cdot\nabla b\cdot\pa_2 b~dxdy-\int_{\Omega}\pa_2b\cdot\nabla u\cdot\pa_2 b~dxdy-\int_{\Omega}b\cdot\nabla\pa_2u\cdot\pa_2b~dxdy\\
	&\triangleq I_1+I_2+I_3+J_1+J_2+J_3.
	\end{split}
	\end{equation}
	According to the divergence-free condition of $v$, after integration by parts, we can obtain 
	\begin{equation*}
	\begin{split}
	I_3+J_3=-\int_{\Omega}b\cdot\nabla\pa_2b\cdot\pa_2u~dxdy-\int_{\Omega}b\cdot\nabla\pa_2u\cdot\pa_2b~dxdy=0.
	\end{split}
	\end{equation*}
	For $I_{1},$ we can write it as
	\begin{align}
	I_{1} &=-\int_{\Omega} \partial_{2} u \cdot \nabla u \cdot \partial_{2} u d x d y \notag\\
	&=-\int_{\Omega} \partial_{2} u^{1} \partial_{1} u^{1} \partial_{2} u^{1} d x d y-\int_{\Omega} \partial_{2} u^{1} \partial_{1} u^{2} \partial_{2} u^{2} d x d y \notag\\
	&\quad-\int_{\Omega} \partial_{2} u^{2} \pa_{2} u^{1} \partial_{2} u^{1} d x d y-\int_{\Omega} \partial_{2} u^{2} \partial_{2} u^{2} \partial_{2} u^{2} d x d y \notag\\
	&=-\int_{\Omega} \partial_{2} u^{1} \partial_{1} u^{2} \partial_{2} u^{2} d x d y-\int_{\Omega} \partial_{2} u^{2} \partial_2 u^{2} \partial_2 u^{2} d x d y.
	\end{align}
	Similar as $A_1$ in section 3, we can bound $I_1$ by 
	\begin{equation} 
	\label{I-1}
	\begin{split}
	I_{1}  \leq C\left\|\partial_{2} u\right\|_{L^2}^{2}\left\|\partial_{1} u\right\|_{L^2}^{2}+\varepsilon\left\|\partial_{1} \partial_{2} u\right\|_{L^2}^{2}.
	\end{split} 
	\end{equation}
	Then we estimate $I_2$, according to the divergence-free condition of $v$,
	\begin{align}
	I_{2}&= -\int_{\Omega} \pa_2 v \cdot \nabla v \cdot\partial_2 u  d x d y \notag\\
	&= -\int_{\Omega} \partial_{2} v^{1} \partial_{1} v^{1} \partial_{2} u^{1} d x d y-\int_{\Omega} \partial_{2} v^{1} \partial_{1} v^{2} \partial_{2} u^{2} d x d y \notag\\
	&\quad-\int_{\Omega} \partial_{2} v^{2} \pa_2v^1 \partial_{2} u^{1} d {x} d y-\int_{\Omega} \partial_{2} v^{2} \partial_{2} v^{2} \partial_{2} u^{2} d x d y \notag\\
	&=-\int_{\Omega} \partial_{2} v^{1} \partial_{1} v^{2} \partial_{2} u^{2} d x d y-\int_{\Omega} \partial_{2} v^{2} \partial_{2} v^{2} \partial_{2} u^{2} d x d y .
	\end{align}
	Similar as $A_2$ in section 3, $I_2$ can be bounded by
	\begin{equation}
	\label{I-2} 
	\begin{split}
	I_{2}  \leq C\left\|\partial_{2} b\right\|_{L^2}^{2}\left\|\partial_{1} (u,b)\right\|_{L^2}^{2}+\varepsilon\left\|\partial_{1} \partial_{2} (u,b)\right\|_{L^2}^{2}.
	\end{split} 
	\end{equation}
	Finally we estimate $J_1$ and $J_2$, we decompose them as
	$$J_1=\int_{\Omega} -\partial_2 u^1\partial_1 v^1\partial_2 v^1-\partial_2 u^1\partial_1 v^2\partial_2 v^2-\partial_2 u^2\partial_2 v^1\partial_2 v^1-\partial_2 u^2\partial_2 v^2\partial_2 v^2~dxdy,$$
	$$J_2=\int_{\Omega} -\partial_2 v^1\partial_1 u^1\partial_2 v^1-\partial_2 v^1\partial_1 u^2\partial_2 v^2-\partial_2 v^2\partial_2 u^1\partial_2 v^1-\partial_2 v^2\partial_2 u^2\partial_2 v^2~dxdy.$$
	Adding them up and using the fact that $\text{div\,}(u,v)=0$, we get
	$$J_1+J_2=\int_{\Omega} -\partial_2 u^1\partial_1 v^2\cdot\partial_2 v^2-\partial_2 v^1\partial_1 u^2\cdot\partial_2 v^2-2\partial_2 u^2\partial_2 v^2\cdot\partial_2 v^2~dxdy.$$
	Also making use of the method which we used to bound $A_3$ and $A_4$ in section 3, we have 
	$$J_1+J_2\leq C\|\partial_1 (u,v)\|_{L^2}^2\|\partial_2  (u,v)\|_{L^2}^2+\varepsilon\|\partial_1\partial_2 (u,v)\|_{L^2}^2.$$
	Inserting the estimates of $I_1$, $I_2$, $J_1$, $J_2$ into \eqref{L2-pa_2uv-1}, and choosing $\varepsilon$ small enough, we have
	\begin{equation}
	\label{L2-pa_2uv-2}
	\begin{split}
	\frac{d}{d t}\|\pa_2(u,v)(t)\|_{L^{2}}^{2}+\|\partial_{1}\pa_2 (u,v)\|_{2}^{2}\leq C\|\partial_1 (u,v)\|_{L^2}^2\|\partial_2  (u,v)\|_{L^2}^2.
	\end{split}
	\end{equation}
	Then by Gr\"onwall inequality and combining with the estimate \eqref{L2-uv}, we can obtain the desired bound for $(\pa_2u,\pa_2b)$, which completes the proof of this lemma.  
	
	In the $H^s~(s\geq2)$ estimate of $(u,v)$, the main difficulty is the control of the nonlinear term such as $(v\cdot\nabla u, v)_{H^s}$ without
	using vertical derivatives of $u$ and $v$. The following lemma is the key argument in the derivation of $H^s$ bound.
	
	\begin{lemma}
		\label{lemma8.2}
		Let $s\geq1$ be a real number. There exists a positive constant $C$ such that, for any divergence-free vector fields $f,g,h$ which satisfy $f,g,h\in H^s$ and $\pa_1f,\pa_1g\in H^s$, 	
		\begin{equation*}
		\begin{split}
		&\quad-\int_{\Omega}\Delta_q(f\cdot\nabla g)\cdot\Delta_q h~dxdy\\
		&\leq C2^{-2qs}b_q\big\{(\|f\|_{L^2}+\|\nabla f^1\|_{L^2})(\|\pa_1g\|_{H^s}\|h\|_{H^s}+\|g\|_{H^s}\|h\|_{H^s})+\|\nabla f^2\|_{L^2}\|g\|_{H^s}\|h\|_{H^s}\\
		&\quad+\|\nabla f^2\|_{L^2}^{\frac12}\|\pa_2\nabla f^2\|_{L^2}^{\frac12}\|g\|_{H^s}^{\frac12}\|\pa_1g\|_{H^s}^{\frac12}\|h\|_{H^s}+\|\pa_2g\|_{L^2}\|\pa_1f\|_{H^s}\|h\|_{H^s}\big\}\\
		&\quad-\int_{\Omega}S_qf\cdot\nabla\Delta_qg\cdot\Delta_q h~dxdy.
		\end{split}
		\end{equation*}
	\end{lemma}
	Lemma \ref{lemma8.2} can be proved along the similar way as the proof of Lemma \ref{three-linear term} in the Appendix A or the proof of Lemma 2.5 in \cite{PZ20}, which used the anisotropic idea in the commutator estimate.
	
	Then we give the $H^s$ estimate for $(u,v)$. Applying $\Delta_q$ to the first and second equation of $\eqref{system-tropical}$, taking the $L^2$ inner product	of
	the resulting equality with $(\Delta_q u, \Delta_q v)$, after integrating by part we obtain
	\begin{align}
	\label{Hs-uv}
	&\quad\frac{1}{2} \frac{d}{d t}\left(\|\Delta_qu(t)\|_{L^{2}}^{2}+\|\Delta_qv(t)\|_{L^{2}}^{2}\right)+\nu\|\partial_{1}\Delta_q u\|_{2}^{2}+\eta\|\partial_{1}\Delta_q v\|_{2}^{2}\notag\\
	&=-\int_{\Omega}\Delta_q(u\cdot\nabla u)\cdot\Delta_q u~dxdy-\int_{\Omega}\Delta_q(v\cdot\nabla v)\cdot\Delta_q u~dxdy\notag\\
	&\quad-\int_{\Omega}\Delta_q(u\cdot\nabla v)\cdot\Delta_q v~dxdy-\int_{\Omega}\Delta_q(v\cdot\nabla u)\cdot\Delta_q v~dxdy\notag\\
	&\triangleq B_1+B_2+B_3+B_4.
	\end{align}
	According to Lemma \ref{lemma8.2} and the divergence-free condition of $u$, $B_1$ can be bounded by
	\begin{equation*}
	\begin{split}
	B_1&=-\int_{\Omega}\Delta_q(u\cdot\nabla u)\cdot\Delta_q u~dxdy\\
	&\leq C2^{-2qs}b_q\big\{(\|u\|_{L^2}+\|\nabla u^1\|_{L^2}+\|\pa_2u\|_{L^2})(\|\pa_1u\|_{H^s}\|u\|_{H^s}+\|u\|_{H^s}^2)\\
	&\quad+\|\nabla u^2\|_{L^2}\|u\|_{H^s}^2+\|\nabla u^2\|_{L^2}^{\frac12}\|\pa_2\nabla u^2\|_{L^2}^{\frac12}\|u\|_{H^s}^{\frac32}\|\pa_1u\|_{H^s}^{\frac12}\big\},
	\end{split}
	\end{equation*}
	and $B_3$ can be bounded by
	\begin{equation*}
	\begin{split}
	B_3&=-\int_{\Omega}\Delta_q(u\cdot\nabla v)\cdot\Delta_q v~dxdy\\
	&\leq C2^{-2qs}b_q\big\{(\|u\|_{L^2}+\|\nabla u^1\|_{L^2})(\|\pa_1v\|_{H^s}\|v\|_{H^s}+\|v\|_{H^s}^2)+\|\nabla u^2\|_{L^2}\|v\|_{H^s}^2\\
	&\quad+\|\nabla u^2\|_{L^2}^{\frac12}\|\pa_2\nabla u^2\|_{L^2}^{\frac12}\|v\|_{H^s}^{\frac32}\|\pa_1v\|_{H^s}^{\frac12}+\|\pa_2v\|_{L^2}\|\pa_1u\|_{H^s}\|v\|_{H^s}\big\}.
	\end{split}
	\end{equation*}
	Similarly, according to the divergence-free condition of $v$, $B_2$ and $B_4$ can be handled by
	\begin{equation*}
	\begin{split}
	B_2+B_4&=-\int_{\Omega}\Delta_q(v\cdot\nabla v)\cdot\Delta_q u~dxdy-\int_{\Omega}\Delta_q(v\cdot\nabla u)\cdot\Delta_q v~dxdy\\
	&\leq C2^{-2qs}b_q\big\{(\|v\|_{L^2}+\|\nabla v^1\|_{L^2})(\|\pa_1v\|_{H^s}\|u\|_{H^s}+\|\pa_1u\|_{H^s}\|v\|_{H^s}+\|v\|_{H^s}\|u\|_{H^s})\\
	&\quad+\|\nabla v^2\|_{L^2}\|v\|_{H^s}\|u\|_{H^s}+\|\pa_2v\|_{L^2}\|\pa_1v\|_{H^s}\|u\|_{H^s}+\|\pa_2u\|_{L^2}\|\pa_1v\|_{H^s}\|v\|_{H^s}\\
	&\quad+\|\nabla v^2\|_{L^2}^{\frac12}\|\pa_2\nabla v^2\|_{L^2}^{\frac12}\big(\|v\|_{H^s}^{\frac12}\|\pa_1v\|_{H^s}^{\frac12}\|u\|_{H^s}+\|u\|_{H^s}^{\frac12}\|\pa_1u\|_{H^s}^{\frac12}\|v\|_{H^s}\big)\big\}.
	\end{split}
	\end{equation*}
	Inserting the estimates of $B_1-B_4$ into \eqref{Hs-uv}, multiplying the resulting inequality by $2^{2qs}$
	and taking summation in $q$, combining with the global $L^2$ bound of $(u,v)$ and $(\pa_1u,\pa_1v)$, one can deduce
	\begin{equation*}
	\label{Hs-uv-1}
	\begin{split}
	&\quad\frac{1}{2} \frac{d}{d t}\left(\|u(t)\|_{H^s}^{2}+\|v(t)\|_{H^s}^{2}\right)+\nu\|\partial_{1} u\|_{H^s}^{2}+\eta\|\partial_{1} v\|_{H^s}^{2}\\
	& \leq 
	C(\|\pa_1u\|_{H^s}\|u\|_{H^s}+\|\pa_1v\|_{H^s}\|v\|_{H^s}+\|\pa_1v\|_{H^s}\|u\|_{H^s}+\|\pa_1u\|_{H^s}\|v\|_{H^s}+\|v\|_{H^s}^2+\|u\|_{H^s}^2)\\
	&\quad+C(\|\nabla u^2\|_{L^2}+\|\nabla v^2\|_{L^2})(\|u\|_{H^s}^2+\|v\|_{H^s}^2)+C(\|\nabla u^2\|_{L^2}^{\frac12}\|\pa_2\nabla u^2\|_{L^2}^{\frac12}+\|\nabla v^2\|_{L^2}^{\frac12}\|\pa_2\nabla v^2\|_{L^2}^{\frac12})\\
	&\quad\times\big(\|u\|_{H^s}^{\frac32}\|\pa_1u\|_{H^s}^{\frac12}+\|v\|_{H^s}^{\frac32}\|\pa_1v\|_{H^s}^{\frac12}+\|v\|_{H^s}^{\frac12}\|\pa_1v\|_{H^s}^{\frac12}\|u\|_{H^s}+\|u\|_{H^s}^{\frac12}\|\pa_1u\|_{H^s}^{\frac12}\|v\|_{H^s}\big)\big\}\\
	&\leq C(1+\|\nabla u^2\|_{L^2}+\|\nabla v^2\|_{L^2}+\|\nabla u^2\|_{L^2}^{\frac23}\|\pa_2\nabla u^2\|_{L^2}^{\frac23}+\|\nabla v^2\|_{L^2}^{\frac23}\|\pa_2\nabla v^2\|_{L^2}^{\frac23})\\
	&\quad\times(\|u\|_{H^s}^2+\|v\|_{H^s}^2)+
	\frac{\nu}{2}\|\pa_1u\|_{L^2}^2+\frac{\eta}{2}\|\pa_1v\|_{L^2}^2,
	\end{split}
	\end{equation*}
	where we have used the Young's inequality in the last step.\\
	Noticing that 
	$$\|\nabla u^2\|_{L^2}\leq\|\pa_1 u^2\|_{L^2}+\|\pa_2 u^2\|_{L^2}\leq \|\pa_1 u^2\|_{L^2}+\|\pa_1 u^1\|_{L^2}\leq C\|\pa_1 u\|_{L^2},$$
	$$\|\nabla v^2\|_{L^2}\leq\|\pa_1 v^2\|_{L^2}+\|\pa_2 v^2\|_{L^2}\leq \|\pa_1 v^2\|_{L^2}+\|\pa_1 v^1\|_{L^2}\leq C\|\pa_1 v\|_{L^2},$$
	and similarly 
	$$\|\pa_2\nabla u^2\|_{L^2}\leq C\|\pa_1\pa_2 u\|_{L^2},\quad\|\pa_2\nabla v^2\|_{L^2}\leq C\|\pa_1\pa_2 v\|_{L^2}.$$
	Then using Gr\"onwall's Lemma and combining with the results $\pa_1u,\pa_1b\in L^2(\R_+;L^2(\Omega))$ and $\pa_1\pa_2u,\pa_1\pa_2b\in L^2(\R_+;L^2(\Omega))$, we can obtain
	\begin{equation*}
	\|u(t)\|_{H^s}^{2}+\|v(t)\|_{H^s}^{2}+\nu\int_0^t\|\partial_{1} u(\tau)\|_{H^s}^{2}~d\tau+\eta\int_0^t\|\partial_{1} v(\tau)\|_{H^s}^{2}~d\tau\leq C(t),
	\end{equation*}
	for ant $t>0$. With this global $H^s~ (s\geq2)$ bound, it is enough to obtain the existence and uniqueness result for system \eqref{system-tropical} using the same way as section 5.

\end{section}

\begin{section}{Appendix}
This appendix provides the complete proof of Lemma \ref{three-linear term}.
\begin{proof}{Proof of Lemma \ref{three-linear term}}
\begin{equation*}
\begin{split}
-\int_{\Omega}\Delta_q(f\cdot \nabla g)\cdot\Delta_q h~dx&=-\int_{\Omega}\Delta_q(f^1\partial_1 g)\cdot\Delta_q h~dxdy
-\int_{\Omega}\Delta_q(f^2\partial_2 g)\cdot\Delta_q h~dxdy\\
&\triangleq P+Q.
\end{split}
\end{equation*}
For $P$, by Bony's decomposition, we can divide it into the following three terms,
\begin{align}
\label{eqlem21}
P=&-\int_{\Omega}\Delta_q(f^1 \partial_1 g)\cdot\Delta_q h~dxdy\notag\\
=&-\sum_{|k-q|\leq2}\int_{\Omega}\Delta_q(S_{k-1}f^1 \Delta_k\partial_1 g)\cdot\Delta_q h~dxdy\notag\\
&-\sum_{|k-q|\leq2}\int_{\Omega}\Delta_q(\Delta_kf^1 S_{k-1}\partial_1  g)\cdot\Delta_q h~dxdy\notag\\
&-\sum_{k\geq q-1}\sum_{|k-l|\leq1}\int_{\Omega}\Delta_q(\Delta_kf^1  \Delta_l\partial_1 g)\cdot\Delta_q h~dxdy\notag\\
\triangleq& P_1+P_2+P_3.
\end{align}
For $P_1$, using the decomposition \eqref{decomposition}, we can rewrite it as
\begin{equation*}
\begin{split}
P_1&=\sum_{|k-q | \leq 2} \int_{\Omega} \Delta_{q}\left(S_{k-1} \widetilde{f}^{1} \partial_{1} \Delta_{k} \widetilde{g}\right)\cdot \Delta_{q} \bar{h}~d x d y\\
&\quad+\sum_{|k-q | \leq 2} \int_{\Omega} \Delta_{q}\left(S_{k-1} \widetilde{f}^{1} \partial_{1} \Delta_{k} \widetilde{g}\right)\cdot \Delta_{q} \widetilde{h}~d x d y\\
&\quad+\sum_{|k-q | \leq 2} \int_{\Omega} \Delta_{q}\left(S_{k-1} \bar{f}^{1} \partial_{1} \Delta_{k} \widetilde{g}\right)\cdot \Delta_{q} \bar{h}~d x d y\\
&\quad+\sum_{|k-q | \leq 2} \int_{\Omega} \Delta_{q}\left(S_{k-1} \bar{f}^{1} \partial_{1} \Delta_{k} \widetilde{g}\right)\cdot \Delta_{q} \widetilde{h}~d x d y\\
&\triangleq P_{11}+P_{12}+P_{13}+P_{14}.
\end{split}
\end{equation*}
For $P_{11}$, by anisotropic H\"older inequality and Poincar\'e inequality \eqref{Poincare},
\begin{equation*}
\begin{aligned}
P_{11} &=\sum_{| k-q | \leq 2} \int_{\Omega}\Delta_q\left(S_{k-1} \tilde{f}^{1} \partial_{1} \Delta_{k} \widetilde{g}\right) \cdot\Delta_{q} \bar{h}~d x d y  \\
&\leq C \sum_{| k-q | \leq 2}\left\|S_{k-1} \widetilde{f}^{1}\right\|_{L^\infty_yL_{x}^{2}}\left\|\partial_{1} \Delta_k\widetilde{g}\right\|_{L^{2}} \left\|\Delta_{q} \bar{h}\right\|_{L_{y}^{2}} \\
& \leq C\sum_{| k-q | \leq 2}\left\|S_{k-1} \widetilde{f}^{1}\right\|_{L^2}^{\frac12}\left\|S_{k-1}\pa_2 \widetilde{f}^{1}\right\|_{L^2}^{\frac12}\left\|\partial_{1} \Delta_k\widetilde{g}\right\|_{L^{2}} \left\|\Delta_{q} \bar{h}\right\|_{L_{y}^{2}} \\
& \leq C\sum_{| k-q | \leq 2}\left\|\pa_1S_{k-1} \widetilde{f}^{1}\right\|_{L^2}^{\frac12}\left\|S_{k-1}\pa_1\pa_2 \widetilde{f}^{1}\right\|_{L^2}^{\frac12}\left\|\partial_{1} \Delta_k\widetilde{g}\right\|_{L^{2}} \left\|\Delta_{q} \bar{h}\right\|_{L_{y}^{2}} \\
& \leq C2^{-2qs}b_q\left\|\pa_1{f}\right\|_{L^2}^{\frac12}\left\|\pa_1\pa_2 {f}\right\|_{L^2}^{\frac12}
\left\|\partial_{1} {g}\right\|_{H^s} \left\| h\right\|_{H^s}.
\end{aligned}
\end{equation*}
For $P_{12}$, along the same method,
\begin{align*}
P_{12} &=\sum_{| k-q | \leq 2} \int_{\Omega}\Delta_q\left(S_{k-1} \tilde{f}^{1} \partial_{1} \Delta_{k} \widetilde{g}\right)\cdot \Delta_{q} \widetilde{h}~d x d y  \\
&\leq C \sum_{| k-q | \leq 2}\left\|S_{k-1} \widetilde{f}^{1}\right\|_{L^\infty_yL_{x}^{2}}\left\|\partial_{1} \Delta_k\widetilde{g}\right\|_{L^{2}} \left\|\Delta_{q} \widetilde{h}\right\|_{L^\infty_xL_{y}^{2}} \\
& \leq C\sum_{| k-q | \leq 2}\left\|S_{k-1} \widetilde{f}^{1}\right\|_{L^2}^{\frac12}\left\|S_{k-1}\pa_2 \widetilde{f}^{1}\right\|_{L^2}^{\frac12}\left\|\partial_{1} \Delta_k\widetilde{g}\right\|_{L^{2}} \left\|\Delta_{q} \widetilde{h}\right\|_{L^{2}}^{\frac12}\left\|\Delta_{q}\pa_1 \widetilde{h}\right\|_{L^{2}}^{\frac12} \\
& \leq C\sum_{| k-q | \leq 2}\left\|S_{k-1} \widetilde{f}^{1}\right\|_{L^2}^{\frac12}\left\|S_{k-1}\pa_1\pa_2 \widetilde{f}^{1}\right\|_{L^2}^{\frac12}\left\|\partial_{1} \Delta_k\widetilde{g}\right\|_{L^{2}} \left\|\Delta_{q} \widetilde{h}\right\|_{L^{2}}^{\frac12}\left\|\Delta_{q}\pa_1 \widetilde{h}\right\|_{L^{2}}^{\frac12}  \\
& \leq C2^{-2qs}b_q\left\|{f}\right\|_{L^2}^{\frac12}\left\|\pa_1\pa_2 {f}\right\|_{L^2}^{\frac12}
\left\|\partial_{1} {g}\right\|_{H^s} \left\| h\right\|_{H^s}^{\frac12}\left\|\pa_1 h\right\|_{H^s}^{\frac12}.
\end{align*}
According to the duality property of the operator $\Delta_q$, 
\begin{equation*}
\begin{split}
P_{13}&=\sum_{|k-q | \leq 2} \int_{\Omega} \Delta_{q}\left(S_{k-1} \bar{f}^{1} \partial_{1} \Delta_{k} \widetilde{g}\right)\cdot \Delta_{q} \bar{h}~d x d y\\
&=\sum_{|k-q | \leq 2} \int_{\Omega} \left(S_{k-1} \bar{f}^{1} \partial_{1} \Delta_{k} \widetilde{g}\right)\cdot \Delta_{q}^2 \bar{h}~ d x d y\\
&=\sum_{|k-q | \leq 2} \int_{\R}S_{k-1} \bar{f}^{1}\Delta_{q}^2 \bar{h}\cdot \bigg(\int_{\mathbb{T}} \partial_{1} \Delta_{k} \widetilde{g}(x,y)~d x\bigg)~d y\\
&=0.
\end{split}
\end{equation*}
Also by anisotropic H\"older inequality, interpolation \eqref{interpolation} and Poincar\'e inequality \eqref{Poincare},
\begin{equation*}
\begin{split}
P_{14} &=\sum_{| k-q | \leq 2} \int_{\Omega}\Delta_q\left(S_{k-1} \bar{f}^{1} \partial_{1} \Delta_{k} \widetilde{g}\right)\cdot \Delta_{q} \widetilde{h}~d x d y  \\
&\leq C \sum_{| k-q | \leq 2}\left\|S_{k-1} \bar{f}^{1}\right\|_{L^\infty_y}\left\|\partial_{1} \Delta_k\widetilde{g}\right\|_{L^{2}} \left\|\Delta_{q} \widetilde{h}\right\|_{L^{2}} \\
& \leq C\sum_{| k-q | \leq 2}\left\|S_{k-1} \bar{f}^{1}\right\|_{L^2}^{\frac12}\left\|S_{k-1}\pa_2 \bar{f}^{1}\right\|_{L^2}^{\frac12}\left\|\partial_{1} \Delta_k\widetilde{g}\right\|_{L^{2}} \left\|\pa_1\Delta_{q} \widetilde{h}\right\|_{L_2} \\
& \leq C2^{-2qs}b_q\left\|{f}\right\|_{L^2}^{\frac12}\left\|\pa_2 {f}\right\|_{L^2}^{\frac12}
\left\|\partial_{1} {g}\right\|_{H^s} \left\| \pa_1h\right\|_{H^s}.
\end{split}
\end{equation*}
Combining the estimates of $P_{11}-P_{14}$, we obtain the estimate for $P_1$ that,
\begin{equation}
\begin{aligned}
P_{1} 
& \leq 
C2^{-2qs}b_q(\left\|\pa_1{f}\right\|_{L^2}\left\|\pa_1\pa_2 {f}\right\|_{L^2}
 \left\| h\right\|_{H^s}^2
 +\left\|{f}\right\|_{L^2}^2\left\|\pa_1\pa_2 {f}\right\|_{L^2}^2
\left\| h\right\|_{H^s}^2 )\\
&\quad+C\varepsilon_02^{-2qs}b_q(\left\|\partial_{1} {g}\right\|_{H^s}^2+ \left\|\pa_1 h\right\|_{H^s}^2)\\
&\quad+2^{-2qs}b_q\left\|{f}\right\|_{L^2}^{\frac12}\left\|\pa_2 {f}\right\|_{L^2}^{\frac12}
(\left\|\partial_{1} {g}\right\|_{H^s}^2+ \left\| \pa_1h\right\|_{H^s}^2).
\end{aligned}
\end{equation} 
Then we estimate $P_2$, also by the decomposition \eqref{decomposition}, we can write $P_2$ into the following three terms,
\begin{align*}
P_2&=-\sum_{|k-q|\leq 2}\int_{\Omega}\Delta_q(\Delta_kf^1 S_{k-1}\partial_1  g)\cdot\Delta_q h~dxdy\\
&=-\sum_{|k-q | \leq 2} \int_{\Omega} \Delta_{q}\left( \Delta_{k}\widetilde{f}^{1} \partial_{1} S_{k-1} {g}\right)\cdot \Delta_{q} {h}~ d x d y\\
&\quad-\sum_{|k-q | \leq 2} \int_{\Omega} \Delta_{q}\left( \Delta_{k}\bar{f}^{1} \partial_{1} S_{k-1} \widetilde{g}\right)\cdot \Delta_{q} \bar{h}~ d x d y\\
&\quad-\sum_{|k-q | \leq 2} \int_{\Omega} \Delta_{q}\left( \Delta_{k}\bar{f}^{1} \partial_{1} S_{k-1} \widetilde{g}\right)\cdot \Delta_{q} \widetilde{h} ~d x d y\\
&\triangleq P_{21}+P_{22}+P_{23}.
\end{align*}
We write that owing to anisotropic H\"older inequality, interpolation \eqref{interpolation} and Poincar\'e inequality \eqref{Poincare},
\begin{equation*}
\begin{aligned}
P_{21}
&\leq C \sum_{| k-q | \leq 2}\left\|\Delta_k \widetilde{f}^{1}\right\|_{L^\infty_xL_{y}^{2}}\left\|\partial_{1} S_{k-1}{g}\right\|_{L^{\infty}_yL^2_x} \left\|\Delta_{q} {h}\right\|_{L^{2}} \\
& \leq C\sum_{| k-q | \leq 2}\left\|\Delta_k \widetilde{f}^{1}\right\|_{L^2}^{\frac12}\left\|\pa_1\Delta_k \widetilde{f}^{1}\right\|_{L^2}^{\frac12}\left\|\partial_{1} S_{k-1}{g}\right\|_{L^{2}}^{\frac12}\left\|\partial_{1}\pa_2 S_{k-1}{g}\right\|_{L^2}^{\frac12}  \left\|\Delta_{q} {h}\right\|_{L^{2}} \\
& \leq C\sum_{| k-q | \leq 2}
\left\|\partial_{1} {g}\right\|_{L^{2}}^{\frac12}\left\|\partial_{1}\pa_2 {g}\right\|_{L^2}^{\frac12}\left\|\pa_1\Delta_k \widetilde{f}^{1}\right\|_{L^2}
 \left\|\Delta_{q} {h}\right\|_{L^{2}}\\ 
& \leq C2^{-2qs}b_q\left\|\pa_1{g}\right\|_{L^2}^{\frac12}\left\|\pa_1\pa_2 {g}\right\|_{L^2}^{\frac12}
\left\|\partial_{1} {f}\right\|_{H^s} \left\| h\right\|_{H^s}.
\end{aligned}
\end{equation*}
Similar as $P_{13}$, it is easy to see that $P_{22}=0$. To bound $P_{23}$, we write that
\begin{align*}
P_{23} &=\sum_{| k-q | \leq 2} \int_{\Omega}\Delta_q\left( \Delta_{k}\bar{f}^{1} S_{k-1}\partial_{1}  \widetilde{g}\right)\cdot\Delta_{q} \widetilde{h}~d x d y  \\
&\leq C \sum_{| k-q | \leq 2}\left\|\Delta_k \bar{f}^{1}\right\|_{L^2_y}\left\|\partial_{1} S_{k-1}\widetilde{g}\right\|_{L^{\infty}_yL^{2}_x} \left\|\Delta_{q} \widetilde{h}\right\|_{L^{2}} \\
& \leq C\sum_{| k-q | \leq 2}\left\|\Delta_k \bar{f}^{1}\right\|_{L^2_y}\left\|\partial_{1} S_{k-1}\widetilde{g}\right\|_{L^2}^{\frac12}\left\|\partial_{1}\pa_2 S_{k-1}\widetilde{g}\right\|_{L^{2}}^{\frac12} \left\|\pa_1\Delta_{q} \widetilde{h}\right\|_{L^2} \\
& \leq C2^{-2qs}b_q\left\|{\pa_1g}\right\|_{L^2}^{\frac12}\left\|\pa_1\pa_2 {g}\right\|_{L^2}^{\frac12}
\left\| {f}\right\|_{H^s} \left\| \pa_1h\right\|_{H^s}.
\end{align*} 
Making use of Young's inequality, we deduce the estimate of $P_2$ that 
\begin{equation*}
\begin{split}
P_{2} 
& \leq
C2^{-2qs}b_q\left\|\pa_1{g}\right\|_{L^2}\left\|\pa_1\pa_2 {g}\right\|_{L^2}(
\left\| {f}\right\|_{H^s}^2+ \left\| h\right\|_{H^s}^2)+
 C\varepsilon_02^{-2qs}b_q(
\left\| \pa_1{f}\right\|_{H^s}^2+ \left\| \pa_1h\right\|_{H^s}^2).
\end{split}
\end{equation*} 
For $P_3$, first we use the decomposition \eqref{decomposition},
\begin{align*}
P_3&=-\sum_{k\geq q-1}\sum_{|k-l|\leq1}\int_{\Omega}\Delta_q(\Delta_kf^1  \Delta_l\partial_1 g)\cdot\Delta_q h~dx dy\\
&=-\sum_{k\geq q-1}\sum_{|k-l|\leq1} \int_{\Omega} \Delta_{q}\left( \Delta_{k}\widetilde{f}^{1} \partial_{1} \Delta_l {g}\right)\cdot \Delta_{q} {h}~d x d y\\
&\quad-\sum_{k\geq q-1}\sum_{|k-l|\leq1} \int_{\Omega} \Delta_{q}\left( \Delta_{k}\bar{f}^{1} \partial_{1} \Delta_l \widetilde{g}\right)\cdot \Delta_{q} \bar{h}~d x d y\\
&\quad-\sum_{k\geq q-1}\sum_{|k-l|\leq1} \int_{\Omega} \Delta_{q}\left( \Delta_{k}\bar{f}^{1} \partial_{1} \Delta_l \widetilde{g}\right)\cdot \Delta_{q} \widetilde{h}~ d x d y\\
&\triangleq P_{31}+P_{32}+P_{33}.
\end{align*}
The term $P_{31}$ can be handled in the same way as $P_{21}$ and utilizing Bernstein inequality,
\begin{equation*}
\begin{aligned}
P_{31}
&\leq C \sum_{k\geq q-1}\sum_{|k-l|\leq1}\left\|\Delta_k \widetilde{f}^{1}\right\|_{L^\infty_xL_{y}^{2}}\left\|\partial_{1} \Delta_l{g}\right\|_{L^2} \left\|\Delta_{q} {h}\right\|_{L^\infty_yL^{2}_x} \\
& \leq C\sum_{k\geq q-1}\sum_{|k-l|\leq1}\left\|\Delta_k \widetilde{f}^{1}\right\|_{L^2}^{\frac12}\left\|\pa_1\Delta_k \widetilde{f}^{1}\right\|_{L^2}^{\frac12}\left\|\partial_{1} \Delta_l{g}\right\|_{L^{2}}  \left\|\Delta_{q} {h}\right\|_{L^{2}}^{\frac12} \left\|\pa_2\Delta_{q} {h}\right\|_{L^{2}}^{\frac12}\\
& \leq C\sum_{k\geq q-1}2^{\frac q2}2^{-\frac k2}\left\|\pa_1\Delta_k \widetilde{f}^{1}\right\|_{L^2}^{\frac12}\left\|\pa_1\nabla\Delta_k \widetilde{f}^{1}\right\|_{L^2}^{\frac12}\left\|\partial_{1} \Delta_k{g}\right\|_{L^{2}}  \left\|\Delta_{q} {h}\right\|_{L^{2}} \\ 
& \leq C2^{-2qs}b_q\left\|\pa_1{f}\right\|_{L^2}^{\frac12}\left\|\pa_1\nabla {f}^1\right\|_{L^2}^{\frac12}
\left\|\partial_{1} {g}\right\|_{H^s} \left\| h\right\|_{H^s}\\
&\leq C2^{-2qs}b_q\left\|\pa_1{f}\right\|_{L^2}^{\frac12}\left\|\pa_1\pa_2{f}\right\|_{L^2}^{\frac12}
\left\|\partial_{1} {g}\right\|_{H^s} \left\| h\right\|_{H^s},
\end{aligned}
\end{equation*}
where we have used the relation
$\nabla f^1=(\pa_1f^1,\pa_2f^1)=(-\pa_2f^2,\pa_2f^1)$ in the last step.\\
Similar as $P_{13}$, it is easy to check that $P_{32}=0$. To bound $P_{33}$, we write that
\begin{equation*}
\begin{aligned}
P_{33}&=-\sum_{k\geq q-1}\sum_{|k-l|\leq1} \int_{\Omega} \Delta_{q}\left( \Delta_{k}\bar{f}^{1} \partial_{1} \Delta_l \widetilde{g}\right)\cdot \Delta_{q} \widetilde{h}~d x d y\\
&\leq C \sum_{k\geq q-1}\sum_{|k-l|\leq1}\left\|\Delta_k \bar{f}^{1}\right\|_{L_{y}^{2}}\left\|\partial_{1} \Delta_l\widetilde{g}\right\|_{L^2} \left\|\Delta_{q} \widetilde{h}\right\|_{L^\infty_yL^{2}_x} \\
&\leq C \sum_{k\geq q-1}\sum_{|k-l|\leq1}\left\|\Delta_k \bar{f}^{1}\right\|_{L_{y}^{2}}^{\frac12} 2^{-\frac k2}\left\|\nabla\Delta_k \bar{f}^{1}\right\|_{L_{y}^{2}}^{\frac12}
\left\|\partial_{1} \Delta_l\widetilde{g}\right\|_{L^2} 2^{\frac q2}\left\|\Delta_{q} \widetilde{h}\right\|_{L^2_yL^{2}_x} \\
&\leq C \sum_{k\geq q-1}\sum_{|k-l|\leq1}2^{\frac q2-\frac k2}\left\|\Delta_k \bar{f}^{1}\right\|_{L_{y}^{2}}^{\frac12} \left\|\pa_1\Delta_k \bar{f}^{1}\right\|_{L_{y}^{2}}^{\frac12}
\left\|\partial_{1} \Delta_l\widetilde{g}\right\|_{L^2} \left\|\Delta_{q} \widetilde{h}\right\|_{L^2}^{\frac12}\left\|\pa_1\Delta_{q} \widetilde{h}\right\|_{L^2}^{\frac12} \\
&\quad+C \sum_{k\geq q-1}\sum_{|k-l|\leq1}2^{\frac q2-\frac k2}\left\|\Delta_k \bar{f}^{1}\right\|_{L_{y}^{2}}^{\frac12} \left\|\pa_2\Delta_k \bar{f}^{1}\right\|_{L_{y}^{2}}^{\frac12}
\left\|\partial_{1} \Delta_l\widetilde{g}\right\|_{L^2} \left\|\pa_1\Delta_{q} \widetilde{h}\right\|_{L^2} \\
& \leq C2^{-2qs}b_q(\left\|{f}\right\|_{L^2}^{\frac12}\left\|\pa_1 {f}^1\right\|_{L^2}^{\frac12}
\left\|\partial_{1} {g}\right\|_{H^s} \left\| h\right\|_{H^s}^{\frac12}\left\| \pa_1h\right\|_{H^s}^{\frac12}+\left\|{f}\right\|_{L^2}^{\frac12}\left\|\pa_2 {f}^1\right\|_{L^2}^{\frac12}
\left\|\partial_{1} {g}\right\|_{H^s} \left\| \pa_1h\right\|_{H^s}).
\end{aligned}
\end{equation*}
Combining with Young's inequality, we obtain the estimate for $P_3$ that,
\begin{equation*}
\begin{aligned}
P_{3}& \leq C2^{-2qs}b_q(\left\|{f}\right\|_{L^2}^{2}\left\|\pa_1 {f}^1\right\|_{L^2}^{2}
 \left\| h\right\|_{H^s}^{2}+\left\|\pa_1{f}\right\|_{L^2}\left\|\pa_1\pa_2 {f}^1\right\|_{L^2}
 \left\| h\right\|_{H^s}^2)\\
 &\quad+C2^{-2qs}b_q\left\|{f}\right\|_{L^2}^{\frac12}\left\|\pa_2 {f}^1\right\|_{L^2}^{\frac12}
(\left\|\partial_{1} {g}\right\|_{H^s}^2+ \left\| \pa_1h\right\|_{H^s}^2)
+C\varepsilon_0 2^{-2qs}b_q(
\left\|\partial_{1} {g}\right\|_{H^s}^2+\left\| \pa_1h\right\|_{H^s}^2).
\end{aligned}
\end{equation*}
Next we estimate $Q$, first we divide it into three parts,
\begin{equation}
\label{eqlem31}
\begin{split}
-\int_{\R^2}\Delta_q(u^2 \partial_2 f)\cdot\Delta_q f~dxdy
=Q_1+Q_2+Q_3,
\end{split}
\end{equation}
with
\begin{equation*}
Q_1=-\sum_{|k-q|\leq2}\int_{\Omega}\Delta_q(S_{k-1}f^2 \Delta_k\partial_2 g)\cdot\Delta_q h~dxdy,
\end{equation*}
\begin{equation*}
Q_2=-\sum_{|k-q|\leq2}\int_{\Omega}\Delta_q(\Delta_kf^2 S_{k-1}\partial_2  g)\cdot\Delta_q h~dxdy
\end{equation*}
and
\begin{equation*}
Q_3=-\sum_{k\geq q-1}\sum_{|k-l|\leq1}\int_{\Omega}\Delta_q(\Delta_kf^2  \Delta_l\partial_2 g)\cdot\Delta_q h~dxdy.
\end{equation*}
Because $f$ satisfies the divergence-free condition and according to the property \eqref{u2=0}, we can rewrite $Q_1$ as
\begin{align*}
Q_1&=-\sum_{|k-q|\leq2}\int_{\Omega}\Delta_q(S_{k-1}\widetilde{f}^2 \Delta_k\partial_2 g)\cdot\Delta_q h~dxdy\\
&=-\sum_{|k-q|\leq2}\int_{\Omega}[\Delta_q, S_{k-1}\widetilde{f}^2\partial_2]\Delta_k g\cdot\Delta_q h~dxdy\\
&\quad-\sum_{|k-q|\leq2}\int_{\Omega}({ S_{k-1}\widetilde{f}^2-S_q\widetilde{f}^2})\partial_2\Delta_q \Delta_kg\cdot\Delta_q h~dxdy\\
&\quad-\int_{\Omega}{S}_{q}\widetilde{f}^2\partial_2\Delta_q g\cdot\Delta_q h~dxdy\\
&\triangleq Q_{11}+Q_{12}+Q_{13}.
\end{align*}
For $Q_{11}$, we decompose it into the following four terms,
\begin{align*}
Q_{11}
&=-\sum_{|k-q|\leq2}\int_{\Omega}[\Delta_q, S_{k-1}\widetilde{f}^2\partial_2]\Delta_k \bar{g} \cdot\Delta_q \bar{h}~dxdy\\
&\quad-\sum_{|k-q|\leq2}\int_{\Omega}[\Delta_q, S_{k-1}\widetilde{f}^2\partial_2]\Delta_k \bar{g} \cdot\Delta_q \widetilde{h}~dxdy\\
&\quad-\sum_{|k-q|\leq2}\int_{\Omega}[\Delta_q, S_{k-1}\widetilde{f}^2\partial_2]\Delta_k \widetilde{g} \cdot\Delta_q \bar{h}~dxdy\\
&\quad-\sum_{|k-q|\leq2}\int_{\Omega}[\Delta_q, S_{k-1}\widetilde{f}^2\partial_2]\Delta_k \widetilde{g} \cdot\Delta_q \widetilde{h}~dxdy,\\
&\triangleq Q_{111}+Q_{112}+Q_{113}+Q_{114},
\end{align*}
where $[X,Y]\triangleq XY-YX$ defining the standard commutator.\\
According to the definition of decomposition \eqref{decomposition}, it is not difficult to check that $Q_{111}=0$. For $Q_{112}$, according to the definition of $\Delta_q$,
\begin{equation*}
\begin{split}
[\Delta_q, S_{k-1}{\widetilde{f}^2\partial_2}]\Delta_k \bar{g}&
=\int_{\Omega}h_q(\mathbf{x}-\mathbf{x}')(S_{k-1}\widetilde{f}^2(\mathbf{x}')\partial_2\Delta_k \bar{g}(\mathbf{x}'))~d\mathbf{x}'\\
&\quad\quad-S_{k-1}\widetilde{f}^2(\mathbf{x})\int_{\Omega}h_q(\mathbf{x}-\mathbf{x}')\partial_2\Delta_k \bar{g}(\mathbf{x}')~d\mathbf{x}'\\
&=\int_{\Omega}h_q(\mathbf{x}-\mathbf{x}')(S_{k-1}\widetilde{f}^2(\mathbf{x}')-S_{k-1}\widetilde{f}^2(\mathbf{x}))\partial_2\Delta_k \bar{g}(\mathbf{x}')~d\mathbf{x}'\\
&=\int_{\Omega}h_q(\mathbf{x}-\mathbf{x}')\int_0^1(\mathbf{x}'-\mathbf{x})\cdot\nabla S_{k-1}\widetilde{f}^2(s\mathbf{x'}+(1-s)\mathbf{x})~ds\partial_2\Delta_k \bar{g}(\mathbf{x}')~d\mathbf{x}'\\
&=\int_{\Omega}\int_0^1h_q(\mathbf{z})\mathbf{z}\cdot\nabla S_{k-1}\widetilde{f}^2(\mathbf{x}-s\mathbf{z})\partial_2\Delta_k \bar{g}(\mathbf{x}-\mathbf{z})~ds d\mathbf{z},
\end{split}
\end{equation*}
where $h_q$ defined as in \eqref{hq}.

Making use of the anisotropic H\"older inequality, interpolation \eqref{interpolation}, Poincar\'e inequality \eqref{Poincare} and Bernstein inequality,
\begin{equation*}
\begin{split}
Q_{112} &=-\sum_{|k-q|\leq2}\int_{\Omega}[\Delta_q, S_{k-1}\widetilde{f}^2\partial_2]\Delta_k \bar{g} \cdot\Delta_q \widetilde{h}~dxdy  \\
&\leq C \sum_{| k-q | \leq 2}2^{-q}\left\|S_{k-1}\nabla \widetilde{f}^{2}\right\|_{L^\infty_yL^2_x}\left\|\partial_{2} \Delta_k\bar{g}\right\|_{L^{2}_y} \left\|\Delta_{q} \widetilde{h}\right\|_{L^{2}} \\
& \leq C\sum_{| k-q | \leq 2}2^{k-q}\left\|S_{k-1}\nabla \widetilde{f}^{2}\right\|_{L^2}^{\frac12}\left\|\pa_2S_{k-1}\nabla \widetilde{f}^{2}\right\|_{L^2}^{\frac12}\left\| \Delta_k{g}\right\|_{L^{2}} \left\|\pa_1\Delta_{q} \widetilde{h}\right\|_{L^{2}} \\
& \leq C2^{-2qs}b_q\left\|{\pa_1f}\right\|_{L^2}^{\frac12}\left\|\pa_1\pa_2 {f}\right\|_{L^2}^{\frac12}
\left\| {g}\right\|_{H^s} \left\| \pa_1h\right\|_{H^s},
\end{split}
\end{equation*} 
where we have used the relation $\nabla \widetilde{f}^2=(\pa_1\widetilde{f}^1,\pa_2\widetilde{f}^2)=(\pa_1\widetilde{f}^1,-\pa_1\widetilde{f}^1)$ in the lase step.\\
The similarly conclusion can also be drawn for $Q_{113}$ and $Q_{114}$ that
\begin{equation*}
\begin{split}
Q_{113} 
& \leq C2^{-2qs}b_q\left\|{\pa_1f}\right\|_{L^2}^{\frac12}\left\|\pa_1\pa_2 {f}\right\|_{L^2}^{\frac12}
\left\| \pa_1{g}\right\|_{H^s} \left\| h\right\|_{H^s},
\end{split}
\end{equation*} 
and
\begin{equation*}
\begin{split}
Q_{114} 
& \leq C2^{-2qs}b_q\left\|{\pa_1f}\right\|_{L^2}^{\frac12}\left\|\pa_1\pa_2 {f}\right\|_{L^2}^{\frac12}
\left\| \pa_1{g}\right\|_{H^s} \left\| h\right\|_{H^s}.
\end{split}
\end{equation*} 
Next we deal with $Q_{12}$, also making use the decomposition \eqref{decomposition},
\begin{equation*}
\begin{split}
Q_{12}&=-\sum_{|k-q|\leq2}\int_{\Omega}({ S_{k-1}\widetilde{f}^2-S_q\widetilde{f}^2})\partial_2\Delta_q \Delta_kg\cdot\Delta_q h~dxdy\\
&=-\sum_{|k-q|\leq2}\int_{\Omega}({ S_{k-1}\widetilde{f}^2-S_q\widetilde{f}^2})\partial_2\Delta_q \Delta_k\bar{g}\cdot\Delta_q \bar{h}~dxdy\\
&\quad-\sum_{|k-q|\leq2}\int_{\Omega}({ S_{k-1}\widetilde{f}^2-S_q\widetilde{f}^2})\partial_2\Delta_q \Delta_k\bar{g}\cdot\Delta_q \widetilde{h}~dxdy\\
&\quad-\sum_{|k-q|\leq2}\int_{\Omega}({ S_{k-1}\widetilde{f}^2-S_q\widetilde{f}^2})\partial_2\Delta_q \Delta_k\widetilde{g}\cdot\Delta_q \bar{h}~dxdy\\
&\quad-\sum_{|k-q|\leq2}\int_{\Omega}({ S_{k-1}\widetilde{f}^2-S_q\widetilde{f}^2})\partial_2\Delta_q \Delta_k\widetilde{g}\cdot\Delta_q \widetilde{h}~dxdy\\
&\triangleq Q_{121}+Q_{122}+Q_{123}+Q_{124}.
\end{split}
\end{equation*}
The same reasoning as $Q_{111}$ implies that $Q_{121}=0$. For $Q_{122}$, by the anisotropic H\"older inequality, Bernstein inequality, interpolation \eqref{interpolation} and Poincar\'e inequality \eqref{Poincare},
\begin{equation*}
\begin{split}
Q_{122} 
&\leq C \sum_{| k-q | \leq 2}2^{-q}\left\|\nabla\Delta_k \widetilde{f}^{2}\right\|_{L^\infty_yL^2_x}\left\|\partial_{2} \Delta_k\bar{g}\right\|_{L^{2}_y} \left\|\Delta_{q} \widetilde{h}\right\|_{L^{2}} \\
& \leq C\sum_{| k-q | \leq 2}2^{k-q}\left\|\nabla\Delta_k \widetilde{f}^{2}\right\|_{L^2}^{\frac12}\left\|\pa_2\nabla\Delta_k \widetilde{f}^{2}\right\|_{L^2}^{\frac12}\left\| \Delta_k{g}\right\|_{L^{2}} \left\|\pa_1\Delta_{q} \widetilde{h}\right\|_{L^{2}} \\
& \leq C2^{-2qs}b_q\left\|{\nabla f^2}\right\|_{L^2}^{\frac12}\left\|\pa_2\nabla {f}^2\right\|_{L^2}^{\frac12}
\left\| {g}\right\|_{H^s} \left\| \pa_1h\right\|_{H^s}\\
& \leq C2^{-2qs}b_q\left\|{\pa_1f}\right\|_{L^2}^{\frac12}\left\|\pa_1\pa_2 {f}\right\|_{L^2}^{\frac12}
\left\| {g}\right\|_{H^s} \left\| \pa_1h\right\|_{H^s},
\end{split}
\end{equation*} 
where we have used the relations
$\pa_2f^2=-\pa_1f^1$ and $\nabla f^2=(\pa_1 f^2, \pa_2 f^2)=(\pa_1 f^2, -\pa_1f^1)$ in the last step.\\
Similar arguments apply to the terms $Q_{123}$ and $Q_{124}$, we can see that
\begin{equation*}
\begin{split}
Q_{123} 
& \leq C2^{-2qs}b_q\left\|{\pa_1f}\right\|_{L^2}^{\frac12}\left\|\pa_1\pa_2 {f}\right\|_{L^2}^{\frac12}
\left\| \pa_1{g}\right\|_{H^s} \left\| h\right\|_{H^s},
\end{split}
\end{equation*} 
and
\begin{equation*}
\begin{split}
Q_{124} 
& \leq C2^{-2qs}b_q\left\|{\pa_1f}\right\|_{L^2}^{\frac12}\left\|\pa_1\pa_2 {f}\right\|_{L^2}^{\frac12}
\left\| \pa_1{g}\right\|_{H^s} \left\| h\right\|_{H^s}.
\end{split}
\end{equation*} 
Gathering the estimates for $Q_{11}$ and $Q_{12}$, we get the bound for $Q_{1}$ that
\begin{equation*}
\begin{split}
Q_{1} 
& \leq C2^{-2qs}b_q\left\|{\pa_1f}\right\|_{L^2}\left\|\pa_1\pa_2 {f}\right\|_{L^2}
 (\left\| g\right\|_{H^s}^2+\left\| h\right\|_{H^s}^2)
 +C\varepsilon_02^{-2qs}b_q
 (\left\| \pa_1{g}\right\|_{H^s}^2+\left\| \pa_1{h}\right\|_{H^s}^2)\\
 &\quad-\int_{\Omega}{S}_{q}\widetilde{f}^2\partial_2\Delta_q g\cdot\Delta_q h~dxdy.
\end{split}
\end{equation*}
Then we estimate $Q_2$. We first decompose $Q_2$ as follows
\begin{align*}
Q_2&=-\sum_{|k-q|\leq2}\int_{\Omega}\Delta_q(\Delta_kf^2 S_{k-1}\partial_2  g)\cdot\Delta_q h~dxdy\\
&=-\sum_{|k-q|\leq2}\int_{\Omega}\Delta_q(\Delta_k\widetilde{f}^2 S_{k-1}\partial_2  g)\cdot\Delta_q h~dxdy\\
&=-\sum_{|k-q|\leq2}\int_{\Omega}\Delta_q(\Delta_k\widetilde{f}^2 S_{k-1}\partial_2  \bar{g})\cdot\Delta_q \bar{h}~dxdy\\
&\quad-\sum_{|k-q|\leq2}\int_{\Omega}\Delta_q(\Delta_k\widetilde{f}^2 S_{k-1}\partial_2  \bar{g})\cdot\Delta_q \widetilde{h}~dxdy\\
&\quad-\sum_{|k-q|\leq2}\int_{\Omega}\Delta_q(\Delta_k\widetilde{f}^2 S_{k-1}\partial_2  \widetilde{g})\cdot\Delta_q \bar{h}~dxdy\\
&\quad-\sum_{|k-q|\leq2}\int_{\Omega}\Delta_q(\Delta_k\widetilde{f}^2 S_{k-1}\partial_2  \widetilde{g})\cdot\Delta_q \widetilde{h}~dxdy\\
&\triangleq Q_{21}+Q_{22}+Q_{23}+Q_{24}.
\end{align*} 
Similar as $Q_{111}$, we can check at once that $Q_{21}=0$. Owing to anisotropic H\"older inequality, Bernstein inequality, interpolation \eqref{interpolation} and Poincar\'e inequality \eqref{Poincare}, we write that
\begin{align*}
Q_{22} &=-\sum_{|k-q|\leq2}\int_{\Omega}\Delta_q(\Delta_k\widetilde{f}^2 S_{k-1}\partial_2  \bar{g})\cdot\Delta_q \widetilde{h}~dxdy \\
&\leq C \sum_{| k-q | \leq 2}\left\|\Delta_k \widetilde{f}^{2}\right\|_{L^{\infty}_yL^{2}_x}\left\|\partial_{2} S_{k-1}\bar{g}\right\|_{L^2_y} \left\|\Delta_{q} \widetilde{h}\right\|_{L^{2}} \\
&\leq C \sum_{| k-q | \leq 2}\left\|\Delta_k \widetilde{f}^{2}\right\|_{L^2}^{\frac12}\left\|\pa_2\Delta_k \widetilde{f}^{2}\right\|_{L^2}^{\frac12}\left\|\partial_{2} S_{k-1}\bar{g}\right\|_{L^2_y} \left\|\Delta_{q} \widetilde{h}\right\|_{L^{2}} \\
&\leq C \sum_{| k-q | \leq 2}\left\|\Delta_k \widetilde{f}^{2}\right\|_{L^2}^{\frac12}\left\|\pa_1\Delta_k \widetilde{f}^{1}\right\|_{L^2}^{\frac12}\left\|\partial_{2} S_{k-1}\bar{g}\right\|_{L^2_y} \left\|\pa_1\Delta_{q} \widetilde{h}\right\|_{L^{2}} \\
& \leq C2^{-2qs}b_q\left\|{\pa_2g}\right\|_{L^2}
\left\| \pa_1{f}\right\|_{H^s} \left\| \pa_1h\right\|_{H^s}.
\end{align*} 
Along the same way,
\begin{equation*}
\begin{split}
Q_{23} &=-\sum_{|k-q|\leq2}\int_{\Omega}\Delta_q(\Delta_k\widetilde{f}^2 S_{k-1}\partial_2  \widetilde{g})\cdot\Delta_q \bar{h}~dxdy\\
&\leq C \sum_{| k-q | \leq 2}\left\|\Delta_k \widetilde{f}^{2}\right\|_{L^{\infty}_yL^{2}_x}\left\|\partial_{2} S_{k-1}\widetilde{g}\right\|_{L^2} \left\|\Delta_{q} \bar{h}\right\|_{L^{2}_y} \\
&\leq C \sum_{| k-q | \leq 2}\left\|\Delta_k \widetilde{f}^{2}\right\|_{L^2}^{\frac12}\left\|\pa_2\Delta_k \widetilde{f}^{2}\right\|_{L^2}^{\frac12}\left\|\pa_1\partial_{2} S_{k-1}\widetilde{g}\right\|_{L^2} \left\|\Delta_{q} \bar{h}\right\|_{L^{2}_y} \\
&\leq C \sum_{| k-q | \leq 2}\left\|\pa_1\Delta_k \widetilde{f}^{2}\right\|_{L^2}^{\frac12}\left\|\pa_1\Delta_k \widetilde{f}^{1}\right\|_{L^2}^{\frac12}\left\|\pa_1\partial_{2} S_{k-1}\widetilde{g}\right\|_{L^2} \left\|\Delta_{q} \bar{h}\right\|_{L^{2}_y} \\
& \leq C2^{-2qs}b_q\left\|{\pa_1\pa_2g}\right\|_{L^2}
\left\| \pa_1{f}\right\|_{H^s} \left\| h\right\|_{H^s}.
\end{split}
\end{equation*} 
In the same manner we can see that
\begin{equation*}
\begin{split}
Q_{24} \leq C2^{-2qs}b_q\left\|{\pa_1\pa_2g}\right\|_{L^2}
\left\| \pa_1{f}\right\|_{H^s} \left\| h\right\|_{H^s}.
\end{split}
\end{equation*} 
Thus we conclude that 
\begin{equation*}
\begin{split}
Q_{2} &\leq C2^{-2qs}b_q\left\|{\pa_1\pa_2g}\right\|_{L^2}^2\left\| h\right\|_{H^s}^2
+ C\varepsilon_02^{-2qs}b_q\left\| \pa_1{f}\right\|_{H^s}^2\\
&\quad+C2^{-2qs}b_q\left\|{\pa_2g}\right\|_{L^2}
(\left\| \pa_1{f}\right\|_{H^s}^2+ \left\| \pa_1h\right\|_{H^s}^2).
\end{split}
\end{equation*} 
Finally we deal with $Q_3$. Similar as $Q_2$, we first divide it into the following four parts,
\begin{align*}
Q_3&=-\sum_{k\geq q-1}\sum_{|k-l|\leq1}\int_{\Omega}\Delta_q(\Delta_kf^2 \Delta_l\partial_2  g)\cdot\Delta_q h~dxdy\\
&=-\sum_{k\geq q-1}\sum_{|k-l|\leq1}\int_{\Omega}\Delta_q(\Delta_k\widetilde{f}^2 \Delta_l\partial_2  g)\cdot\Delta_q h~dxdy\\
&=-\sum_{k\geq q-1}\sum_{|k-l|\leq1}\int_{\Omega}\Delta_q(\Delta_k\widetilde{f}^2 \Delta_l\partial_2  \bar{g})\cdot\Delta_q \bar{h}~dxdy\\
&\quad-\sum_{k\geq q-1}\sum_{|k-l|\leq1}\int_{\Omega}\Delta_q(\Delta_k\widetilde{f}^2 \Delta_l\partial_2  \bar{g})\cdot\Delta_q \widetilde{h}~dxdy\\
&\quad-\sum_{k\geq q-1}\sum_{|k-l|\leq1}\int_{\Omega}\Delta_q(\Delta_k\widetilde{f}^2 \Delta_l\partial_2  \widetilde{g})\cdot\Delta_q \bar{h}~dxdy\\
&\quad-\sum_{k\geq q-1}\sum_{|k-l|\leq1}\int_{\Omega}\Delta_q(\Delta_k\widetilde{f}^2 \Delta_l\partial_2  \widetilde{g})\cdot\Delta_q \widetilde{h}~dxdy\\
&\triangleq Q_{31}+Q_{32}+Q_{33}+Q_{34}.
\end{align*}
A trivial verification shows that $Q_{31}=0$. For $Q_{32}$, we can handle it by
\begin{equation*}
\begin{aligned}
Q_{32}&=-\sum_{k\geq q-1}\sum_{|k-l|\leq1}\int_{\Omega}\Delta_q(\Delta_k\widetilde{f}^2 \Delta_l\partial_2  \bar{g})\cdot\Delta_q \widetilde{h}~dxdy\\
&\leq C \sum_{k\geq q-1}\sum_{|k-l|\leq1}\left\|\Delta_k \widetilde{f}^{2}\right\|_{L^2}\left\|\partial_{2} \Delta_l\bar{g}\right\|_{L^2_y} \left\|\Delta_{q} \widetilde{h}\right\|_{L^\infty_yL^{2}_x} \\
& \leq C\sum_{k\geq q-1}\sum_{|k-l|\leq1}\left\|\Delta_k \widetilde{f}^{2}\right\|_{L^2}^{\frac12}2^{-\frac k2}\left\|\nabla\Delta_k \widetilde{f}^{2}\right\|_{L^2}^{\frac12}\left\|\partial_{2} \Delta_l\bar{g}\right\|_{L^{2}_y}  2^{\frac q2}\left\|\Delta_{q} \widetilde{h}\right\|_{L^{2}}\\
& \leq C\sum_{k\geq q-1}2^{\frac q2}2^{-\frac k2}\left\|\pa_1\Delta_k \widetilde{f}^{2}\right\|_{L^2}^{\frac12}\left\|\pa_1\Delta_k \widetilde{f}\right\|_{L^2}^{\frac12}\left\|\partial_{2} {g}\right\|_{L^{2}}  \left\|\pa_1\Delta_{q} \widetilde{h}\right\|_{L^{2}} \\ 
& \leq C2^{-2qs}b_q\left\|\pa_2{g}\right\|_{L^2}
\left\|\partial_{1} {f}\right\|_{H^s} \left\| \pa_1h\right\|_{H^s},
\end{aligned}
\end{equation*}
where we have used the relation $\nabla \widetilde{f}^2=(\pa_1\widetilde{f}^2,\pa_2\widetilde{f}^2)=(\pa_1\widetilde{f}^2,-\pa_1\widetilde{f}^1)$.\\
$Q_{33}$ can be bounded similarly,
\begin{equation*}
\begin{aligned}
Q_{33}&=-\sum_{k\geq q-1}\sum_{|k-l|\leq1}\int_{\Omega}\Delta_q(\Delta_k\widetilde{f}^2 \Delta_l\partial_2  \widetilde{g})\cdot\Delta_q \bar{h}~dxdy\\
&\leq C \sum_{k\geq q-1}\sum_{|k-l|\leq1}\left\|\Delta_k \widetilde{f}^{2}\right\|_{L^2}\left\|\partial_{2} \Delta_l\widetilde{g}\right\|_{L^2} \left\|\Delta_{q} \bar{h}\right\|_{L^\infty_y} \\
& \leq C\sum_{k\geq q-1}\sum_{|k-l|\leq1}\left\|\Delta_k \widetilde{f}^{2}\right\|_{L^2}^{\frac12}2^{-\frac k2}\left\|\nabla\Delta_k \widetilde{f}^{2}\right\|_{L^2}^{\frac12}\left\|\partial_{2} \Delta_l\widetilde{g}\right\|_{L^{2}}  2^{\frac q2}\left\|\Delta_{q} \bar{h}\right\|_{L^{2}_y}\\
& \leq C\sum_{k\geq q-1}2^{\frac q2}2^{-\frac k2}\left\|\pa_1\Delta_k \widetilde{f}\right\|_{L^2}\left\|\pa_1\partial_{2} \widetilde{g}\right\|_{L^{2}}  \left\|\Delta_{q} \bar{h}\right\|_{L^{2}_y} \\ 
& \leq C2^{-2qs}b_q\left\|\pa_1\pa_2{g}\right\|_{L^2}
\left\|\partial_{1} {f}\right\|_{H^s} \left\| h\right\|_{H^s}.
\end{aligned}
\end{equation*}
The same estimate remains valid for $Q_{34}$ that
\begin{equation*}
\begin{aligned}
Q_{34}\leq C2^{-2qs}b_q\left\|\pa_1\pa_2{g}\right\|_{L^2}
\left\|\partial_{1} {f}\right\|_{H^s} \left\| h\right\|_{H^s}.
\end{aligned}
\end{equation*}
Using Young's inequality, we can get the estimate for $Q_3$ that
\begin{equation*}
\begin{aligned}
Q_{3}&\leq C2^{-2qs}b_q\left\|\pa_1\pa_2{g}\right\|_{L^2}^2\left\| h\right\|_{H^s}^2+C\varepsilon_02^{-2qs}b_q\left\|\partial_{1} {f}\right\|_{H^s}^2\\
&\quad +C2^{-2qs}b_q\left\|\pa_2{g}\right\|_{L^2}
(\left\|\partial_{1} {f}\right\|_{H^s}^2+ \left\| \pa_1h\right\|_{H^s}^2).
\end{aligned}
\end{equation*}
Taking the estimates for $Q_{1}-Q_{3}$, $P_{1}-P_{3}$ into account, finally we can obtain
\begin{equation}
\begin{aligned}
&\quad-\int_{\Omega}\Delta_q(f\cdot \nabla g)\cdot\Delta_q h~dxdy\\
& \leq 
C2^{-2qs}b_q(\left\|\pa_1{f}\right\|_{L^2}\left\|\pa_1\pa_2 {f}\right\|_{L^2}+\left\|{f}\right\|_{L^2}^2\left\|\pa_1 {f}\right\|_{L^2}^2+\left\|{f}\right\|_{L^2}^2\left\|\pa_1\pa_2 {f}\right\|_{L^2}^2\\
&\quad+\left\|\pa_1{g}\right\|_{L^2}\left\|\pa_1\pa_2 {g}\right\|_{L^2}+\left\|{\pa_1\pa_2g}\right\|_{L^2}^2)\times(\left\| {f}\right\|_{H^s}^2+ \left\| g\right\|_{H^s}^2+\left\| h\right\|_{H^s}^2 ) \\
&\quad+2^{-2qs}b_q(\left\|{f}\right\|_{L^2}^{\frac12}\left\|\pa_2 {f}\right\|_{L^2}^{\frac12}+\left\|{\pa_2g}\right\|_{L^2})\times
(\left\| \pa_1{f}\right\|_{H^s}^2+\left\|\partial_{1} {g}\right\|_{H^s}^2+ \left\| \pa_1h\right\|_{H^s}^2)\\
&\quad+C\varepsilon_02^{-2qs}b_q(\left\| \pa_1{f}\right\|_{H^s}^2+\left\|\partial_{1} {g}\right\|_{H^s}^2+ \left\|\pa_1 h\right\|_{H^s}^2)-\int_{\Omega}{S}_{q}\widetilde{f}^2\partial_2\Delta_q g\cdot\Delta_q h~dxdy,
\end{aligned}
\end{equation} 
which completes the proof of this lemma.
\end{proof}
	
\end{section}

\vskip .4in

\section*{Acknowledgements}
The authors would like to thank Professor Jiahong Wu for helpful discussion.
M. Paicu is partially supported by the Agence Nationale de la Recherche, Project IFSMACS, grant ANR-15-CE40-0010. N. Zhu was partially supported by the National Natural Science Foundation of China (Grant No. 11771043 and No. 11771045).

\end{spacing}

\end{document}